\documentclass[12pt,twoside,reqno]{amsart}
\usepackage{mathptmx, amsmath, amssymb, amsfonts, amsthm, enumerate, mathrsfs}
\usepackage{xcolor}
\usepackage[colorlinks=true,citecolor=blue]{hyperref}

\hypersetup{
  pdftitle={Directional Optimality Conditions for Optimization Problems in Asplund Spaces},
  pdfauthor={Weihao Mao and Jane J. Ye}
}

\newtheorem{theorem}{Theorem}[section]
\newtheorem{lemma}[theorem]{Lemma}
\newtheorem{proposition}[theorem]{Proposition}
\newtheorem{corollary}[theorem]{Corollary}

\theoremstyle{definition}
\newtheorem{definition}[theorem]{Definition}

\newtheorem{example}[theorem]{Example}

\newtheorem{remark}[theorem]{Remark}
\numberwithin{equation}{section}

\begin{document}
\title[Directional Optimality Conditions for Optimization Problems in Asplund Spaces] 
{Directional Optimality Conditions for Optimization Problems in Asplund Spaces}

\author[W. Mao]{Weihao Mao}
\address{Department of Mathematics and Statistics, University of Victoria, Victoria, BC V8W 2Y2, Canada}
\email{mwhaea123456@gmail.com}

\author[J. J. Ye]{Jane J. Ye}
\address{Department of Mathematics and Statistics, University of Victoria, Victoria, BC V8W 2Y2, Canada}
\email{janeye@uvic.ca}
\thanks{Corresponding author: Jane J. Ye.}


\begin{abstract}
This paper develops directional necessary optimality conditions for constrained optimization problems in Asplund spaces. Under directional metric subregularity, we first derive a directional optimality condition in terms of limiting subdifferentials. We then introduce sufficient conditions for directional metric subregularity and establish their relationships with directional pseudo-normality and quasi-normality. For systems with joint constraints, we further propose a joint criterion for directional metric subregularity and use it to obtain necessary optimality conditions. The results not only extend several finite-dimensional directional constructions to an Asplund-space setting, but also yield conclusions that are new even in finite dimensions.
\end{abstract}

\maketitle

\section{Introduction}
In this paper, we investigate a constrained optimization model of the form
\begin{equation}\label{eq:basic_model}
    \min_{x \in X} \quad f(x) \qquad \text{s.t.} \quad P(x) \in \Lambda,
\end{equation}
where $X$ and $Y$ are Asplund spaces, $f: X \to \mathbb{R}$ is a locally Lipschitz function, $P: X \to Y$ is a continuous mapping, and $\Lambda \subset Y$ is a closed set. By defining the set-valued mapping $G(x) := P(x) - \Lambda$, the constraint $P(x) \in \Lambda$ can be equivalently expressed as $0 \in G(x)$. 
This formulation serves as a unified framework that encompasses various problems, such as nonconvex optimization problems (see, e.g., \cite{bonnans2013perturbation}) and mathematical programs with geometric constraints (see, e.g., \cite{guo2013mathematical}).

Traditionally, necessary optimality conditions are derived under standard constraint qualifications (CQs). However, these CQs fail in many cases. This limitation motivates the development of directional variational analysis, which characterizes the system behavior along specific directions. Requiring regularity only along these directions yields sharper characterizations of local minimizers.

While directional metric subregularity and its applications to optimality conditions are well-developed in finite dimensions, extending these results to infinite-dimensional Banach spaces presents significant analytical hurdles. Only a few works address this setting, often under restrictive assumptions \cite{gfrerer2013directional,gfrerer2014metric,long2017calculus,long2017calculusb}. The primary difficulty is that sufficient conditions for directional metric subregularity typically rely on the norm compactness of bounded sets \cite{bai2019directional}, a property that fails in infinite dimensions.

Furthermore, the difference between the weak\(^*\) topology and the norm topology on the dual space $Y^*$ in infinite-dimensional settings prevents the direct generalization of the techniques in \cite{gfrerer2014metric,bai2019directional}. In the limiting process for optimality conditions, this causes the ``vanishing multiplier problem,'' in which a sequence of dual elements has unit norm but a vanishing weak\(^*\) limit. 

To overcome these issues, \cite{gfrerer2013directional} utilized coderivatives and partial sequential normal compactness, while \cite{long2017calculus,long2017calculusb} introduced directional inner semicompactness. Although significant, these results are often formulated via abstract coderivatives, which obscures the derivation of necessary conditions in the standard limiting-stationarity form using directional subdifferentials. Moreover, some assumptions of these results remain difficult to verify in applications.

Motivated by these challenges, this paper presents several developments, summarized as follows:
\begin{itemize}
    \item We derive necessary optimality conditions for \eqref{eq:basic_model} in Asplund spaces by employing the directional metric subregularity and the weak\(^*\) strict Lipschitz property (Theorem~\ref{thm:further_optimality_mscq}). These conditions are characterized via scalarization of coderivatives, giving a more computable form of the results in \cite{gfrerer2013directional}. 
    \item To facilitate the verification of directional metric subregularity, we consider the weak sufficient condition for metric subregularity, directional pseudo-normality and directional quasi-normality. Under directional PSNC, these conditions are sufficient for directional metric subregularity in Asplund spaces (Theorem~\ref{thm:sufficient_condition_for_mscq}, Corollary~\ref{cor:dpn_to_mscq} and  Corollary~\ref{cor:dqn_to_mscq}). A joint sequential qualification of the same type is then established for joint constraint systems (Theorem~\ref{thm:directional_mscq_intersection}, Corollary~\ref{coro:directional_mscq_intersection_finite} and Corollary~\ref{coro:intersection_mscq_finite_differentiable}. These criteria extend \cite{bai2019directional,gfrerer2017new} to Asplund spaces, and some of them are new even in finite dimensions.
    \item We apply our theoretical framework to derive directional optimality conditions for optimization problems (Theorem~\ref{thm:optimality_condition_ocpec}). Compared to classical results in \cite{gfrerer2014metric,long2017calculusb}, these conditions are more concrete and tractable. 
\end{itemize}

The remainder of this paper is organized as follows. Section~\ref{sec:preliminary} reviews the preliminaries of variational analysis and introduces their directional variants. Section~\ref{sec:optimality_condition} establishes necessary optimality conditions under directional metric subregularity. Section~\ref{sec:sufficient_condition_for_mscq} develops sufficient conditions for directional metric subregularity. Section~\ref{sec:applications} demonstrates applications to optimization problems. Section~\ref{sec:conclusions} summarizes the main results.

\subsection*{Notation}
Unless otherwise specified, $X$ and $Y$ are Banach spaces, with topological duals $X^*$ and $Y^*$. The duality pairing between $X^*$ and $X$ is written $\langle \cdot, \cdot \rangle$. We use $\|\cdot\|$ for the norm in the space under consideration. For a Banach space $Z$, we denote by $\mathbb{B}_Z$ and $\mathbb{S}_Z$ the closed unit ball and the unit sphere of $Z$, respectively; in particular, $\mathbb{B}_{X^*}$ and $\mathbb{S}_{X^*}$ stand for the closed unit ball and the unit sphere of $X^*$. When the underlying space is clear from the context, we simply write $\mathbb{B}$ and $\mathbb{B}^*$ (resp.\ $\mathbb{S}$ and $\mathbb{S}^*$). Weak\(^*\) convergence is denoted by $\xrightarrow{w^*}$.

The distance from a point $x\in X$ to a set $\Omega\subset X$ is $\operatorname{dist}(x,\Omega):=\inf_{z\in\Omega}\|x-z\|$, with the convention $\operatorname{dist}(x,\emptyset)=+\infty$. The indicator function of $\Omega$ is $\delta_\Omega$, equal to $0$ on $\Omega$ and $+\infty$ outside $\Omega$. For a set-valued mapping $F:X\rightrightarrows Y$, the graph and the domain are
\[
    \mathrm{gph}\, F := \bigl\{ (x,y)\in X\times Y \;\big|\; y\in F(x) \bigr\}, \qquad
    \mathrm{dom}\, F := \bigl\{ x\in X \;\big|\; F(x)\neq\emptyset \bigr\}.
\]
The symbol $\limsup$ applied to a family of sets stands for the sequential Painlev\'e--Kuratowski outer limit; when the sets lie in a dual space, the outer limit is understood in the weak\(^*\) topology.

\section{Preliminaries}\label{sec:preliminary}

In this section, we recall several fundamental notions and properties from variational analysis that are indispensable for our subsequent developments. Beyond the standard concepts, we introduce their directional counterparts and establish several key properties essential for our framework. Many of the results presented herein serve as infinite-dimensional extensions of the finite-dimensional analogues developed in \cite{gfrerer2013directional,benko2019calculus,Bai2023Directional}. 

\subsection{Tangent Cones, Normal Cones and Coderivatives}\label{subsec:prelim_variational}
Throughout this paper, we primarily utilize the Fr\'echet and Mordukhovich (limiting) normal cones, along with their associated subdifferentials, coderivatives, and directional variants.  

Let $X$ and $Y$ be Banach spaces, and let $\Omega \subset X$ be a nonempty subset. For any $\bar{x} \in \Omega$, the \emph{contingent cone} (also known as the Bouligand tangent cone) to $\Omega$ at $\bar{x}$ is defined by
\begin{equation*}
    T(\bar{x}; \Omega) := \left\{ u \in X \;\middle|\; \exists\, t_k \downarrow 0, \, \{x_k\} \subset \Omega \text{ s.t. } \frac{x_k - \bar{x}}{t_k} \to u \right\}.
\end{equation*}
It is well known that $T(\bar{x}; \Omega)$ is a closed cone, though not necessarily convex.

For a fixed $\varepsilon \ge 0$, the set of \emph{$\varepsilon$-normals} to $\Omega$ at $\bar{x} \in \Omega$ is defined as
\begin{equation*}
    \widehat{N}_{\varepsilon}(\bar{x}; \Omega) := \left\{ x^* \in X^* \;\middle|\; \limsup_{x \xrightarrow{\Omega} \bar{x}, \, x \neq \bar{x}} \frac{\langle x^*, x - \bar{x} \rangle}{\|x - \bar{x}\|} \le \varepsilon \right\},
\end{equation*}
where $x \xrightarrow{\Omega} \bar{x}$ signifies that $x \to \bar{x}$ with $x \in \Omega$. By convention, we set $\widehat{N}_{\varepsilon}(\bar{x}; \Omega) = X^*$ if $\Omega = \{\bar{x}\}$, and $\widehat{N}_{\varepsilon}(\bar{x}; \Omega) = \emptyset$ if $\bar{x} \notin \Omega$. In the case where $\varepsilon = 0$, the set $\widehat{N}(\bar{x}; \Omega) := \widehat{N}_0(\bar{x}; \Omega)$ is called the \emph{Fr\'echet normal cone} (or \emph{prenormal cone}) to $\Omega$ at $\bar{x}$.

The \emph{Mordukhovich} (limiting) normal cone to $\Omega$ at $\bar{x}$ is defined as
\begin{equation*}
    N(\bar{x};\Omega) := \left\{ x^* \in X^* \;\middle|\; \exists\, \varepsilon_k \downarrow 0, \, x_k \xrightarrow{\Omega} \bar{x}, \, x_k^* \xrightarrow{w^*} x^*, \, x_k^* \in \widehat{N}_{\varepsilon_k}(x_k;\Omega) \right\}.
\end{equation*}
When $\Omega$ is convex, both $\widehat{N}(\bar{x}; \Omega)$ and $N(\bar{x}; \Omega)$ reduce to the normal cone in the sense of convex analysis. Let $\bar{u} \in X$ be a given direction. The \emph{directional Mordukhovich normal cone} to $\Omega$ at $\bar{x} \in \Omega$ in the direction $\bar{u}$ is defined by
\begin{equation*}
    N(\bar{x}; \Omega; \bar{u}) := \limsup_{\varepsilon, t \downarrow 0, \, u \to \bar{u}} \widehat{N}_{\varepsilon}(\bar{x} + tu; \Omega),
\end{equation*}
which means that $x^* \in N(\bar{x}; \Omega; \bar{u})$ if and only if there exist sequences $\varepsilon_k \downarrow 0$, $t_k \downarrow 0$, $u_k \to \bar{u}$, and $x_k^* \xrightarrow{w^*} x^*$ such that $\bar{x} + t_k u_k \in \Omega$ and $x_k^* \in \widehat{N}_{\varepsilon_k}(\bar{x} + t_k u_k; \Omega)$ for all $k \in \mathbb{N}$. 
When $X$ is an Asplund space and $\Omega$ is locally closed around $\bar{x}$, the $\varepsilon$-dependence can be omitted (cf. \cite[Theorem~2.34]{Mordukhovich2006variational}), leading to the following simplified representation:
\begin{equation*}
    N(\bar{x}; \Omega; \bar{u}) = \limsup_{t \downarrow 0, \, u \to \bar{u}} \widehat{N}(\bar{x} + tu; \Omega).
\end{equation*}
It follows directly from the definition that $N(\bar{x}; \Omega; \bar{u}) = \emptyset$ when $\bar{u} \notin T(\bar{x}; \Omega)$. Moreover, the directional normal cone satisfies the following properties:
\begin{equation*}
\begin{gathered}
    N(\bar{x}; \Omega; \lambda \bar{u}) = N(\bar{x}; \Omega; \bar{u}) \quad \forall \lambda > 0, \\
    N(\bar{x}; \Omega; 0) = N(\bar{x}; \Omega), \qquad
    N(\bar{x}; \Omega; \bar{u}) \subset N(\bar{x}; \Omega).
\end{gathered}
\end{equation*}
In general, the inclusion $N(\bar{x}; \Omega; \bar{u}) \subset N(\bar{x}; \Omega)$ may be strict (see \cite{long2017calculus}). Analogous to the non-directional case, the directional normal cone satisfies the product rule as established in \cite[(3.14) and Theorem~3.7]{long2017calculus}.  

\begin{proposition} \label{prop:dir_normal_product}
Let $\Omega_i \subset X_i$ $(i=1, 2)$ be nonempty subsets and $\bar{x} = (\bar{x}_1, \bar{x}_2) \in \Omega := \Omega_1 \times \Omega_2$. For any direction $\bar{u} = (\bar{u}_1, \bar{u}_2) \in X_1 \times X_2$, the following inclusion holds:
\begin{equation}
    N(\bar{x}; \Omega; \bar{u}) \subset N(\bar{x}_1; \Omega_1; \bar{u}_1) \times N(\bar{x}_2; \Omega_2; \bar{u}_2).
\end{equation}
\end{proposition}

For a set-valued mapping $F: X \rightrightarrows Y$ between Banach spaces, the \emph{graphical derivative} of $F$ at $(\bar{x}, \bar{y}) \in \mathrm{gph}\, F$ in the direction $\bar{u}$ is defined as  
\begin{equation*}
    \begin{aligned}
        \xi \in DF(\bar{x} \mid \bar{y})(\bar{u}) \iff (\bar{u}, \xi) \in T\bigl((\bar{x}, \bar{y}); \mathrm{gph}\, F\bigr).
    \end{aligned}
\end{equation*}
When $F$ is single-valued, we write $DF(\bar{x})(\bar{u})$ for $DF(\bar{x} \mid F(\bar{x}))(\bar{u})$.

The Fr\'echet coderivative and the \emph{normal coderivative} of $F$ at $(\bar{x}, \bar{y}) \in \mathrm{gph}\, F$ are the set-valued mappings $\widehat{D}_N^* F(\bar{x}, \bar{y}): Y^* \rightrightarrows X^*$ and $D_N^* F(\bar{x}, \bar{y}): Y^* \rightrightarrows X^*$ defined by
\begin{align*}
    \widehat{D}_N^* F(\bar{x}, \bar{y})(y^*) &:= \left\{ x^* \in X^* \;\middle|\; (x^*, -y^*) \in \widehat{N}\big((\bar{x}, \bar{y}); \mathrm{gph}\, F\big) \right\}, \quad \forall y^* \in Y^*;\\
    D_N^* F(\bar{x}, \bar{y})(y^*) &:= \left\{ x^* \in X^* \;\middle|\; (x^*, -y^*) \in N\big((\bar{x}, \bar{y}); \mathrm{gph}\, F\big) \right\}, \quad \forall y^* \in Y^*.
\end{align*}
Correspondingly, for any \(y^*\in Y^*\), the \emph{directional limiting
coderivative} of \(F\) at \((\bar{x},\bar{y})\) in a direction
\((\bar{u},\bar{v})\in X\times Y\) is defined by
\begin{equation*}
    D_N^* F\big((\bar{x}, \bar{y}); (\bar{u}, \bar{v})\big)(y^*) := \left\{ x^* \in X^* \;\middle|\; (x^*, -y^*) \in N\big((\bar{x}, \bar{y}); \mathrm{gph}\, F; (\bar{u}, \bar{v})\big) \right\}.
\end{equation*}
When $F$ is a single-valued mapping $\varphi$, we simplify the notation by writing $D_N^* \varphi\big(\bar{x}; (\bar{u}, \bar{v})\big)$. Furthermore, when $\varphi$ is Hadamard directionally differentiable at $\bar{x}$ in the direction $\bar{u}$, this notation can be further streamlined to $D_N^* \varphi(\bar{x}; \bar{u})$, as the direction $\bar{v} = \varphi'_H(\bar{x}; \bar{u})$ is uniquely determined as defined below. 

\begin{definition}\label{defn:hadamard_derivative}
Consider a mapping \(\varphi:X\to Y\) between Banach spaces. If the
following limit exists, it is called the \emph{Hadamard directional
derivative} of \(\varphi\) at \(\bar{x}\) along \(\bar{u}\):
\[
    \varphi'_H(\bar{x};\bar{u})
    :=\lim_{t\downarrow0,\,u'\to\bar{u}}
    \frac{\varphi(\bar{x}+tu')-\varphi(\bar{x})}{t}.
\]
\end{definition}

\subsection{Subdifferentials and Directional Subdifferentials}

In this subsection, we present the fundamental concepts of directional differentiability and subdifferentials. To describe the local behaviors along specified directions, we begin by introducing the notion of a \emph{directional neighborhood}.

\begin{definition}\label{defn:directional_neighborhood}
Fix $\bar{x},d\in X$, where $X$ is a Banach space. Given $\varepsilon,\delta>0$, set
\begin{equation*}
    \mathcal{V}_{\varepsilon,\delta}(\bar{x}; d) := \bar{x} + \left\{ z \in \varepsilon \mathbb{B}_X \;\middle|\; \big\| \|d\|z - \|z\|d \big\| \leq \delta \|z\| \|d\| \right\}.
\end{equation*}
Any set of this form will be referred to as a \emph{directional neighborhood} of $\bar{x}$ in $d$. When the parameters $\varepsilon$ and $\delta$ are immaterial, we simply write $\mathcal{V}(\bar{x}; d)$.
\end{definition}

Next, we recall the definitions of subdifferentials. Let $\varphi: X \to \overline{\mathbb{R}}:=\mathbb{R}\cup\{\pm\infty\}$ be proper, and let $\bar{x} \in \mathrm{dom}\, \varphi := \{x \in X \mid \varphi(x) < +\infty\}$. For any $\varepsilon \ge 0$, the \emph{$\varepsilon$-subdifferential} of $\varphi$ at $\bar{x}$ is defined by
\begin{equation*}
    \widehat{\partial}_{\varepsilon}\varphi(\bar{x}) := \left\{ x^* \in X^* \;\middle|\; \liminf_{x \to \bar{x}, \, x \neq \bar{x}} \frac{\varphi(x) - \varphi(\bar{x}) - \langle x^*, x - \bar{x} \rangle}{\|x - \bar{x}\|} \geq -\varepsilon \right\}.
\end{equation*}
When $\varepsilon = 0$, this set reduces to the \emph{Fr\'echet subdifferential}, denoted by $\widehat{\partial}\varphi(\bar{x})$. The \emph{Mordukhovich (limiting) subdifferential} $\partial \varphi(\bar{x})$ is defined via the normal cone to the epigraph:
\begin{equation*}
    \partial \varphi(\bar{x}) := \left\{ x^* \in X^* \;\middle|\; (x^*, -1) \in N\big((\bar{x}, \varphi(\bar{x})); \mathrm{epi}\, \varphi\big) \right\}. 
\end{equation*}

The \emph{directional Mordukhovich subdifferential} $\partial \varphi(\bar{x}; \bar{u})$ of $\varphi$ at $\bar{x} \in \mathrm{dom}\, \varphi$ in the direction $\bar{u} \in X$ is defined by
\begin{equation} \label{defn:dir_subdiff}
\begin{aligned}
    \partial \varphi(\bar{x}; \bar{u}) := \big\{ x^* \in X^* \;\mid\; &\exists\, \varepsilon_k \downarrow 0, \, t_k \downarrow 0, \, u_k \to \bar{u} \text{ s.t. } \bar{x} + t_k u_k \in \mathrm{dom}\,\varphi,  \\
    x_k^* \xrightarrow{w^*} x^*, \, &\varphi(\bar{x} + t_k u_k) \to \varphi(\bar{x}) \text{ and } x_k^* \in \widehat{\partial}_{\varepsilon_k}\varphi(\bar{x} + t_k u_k) \big\}.
\end{aligned}
\end{equation}

Furthermore, when $X$ is an Asplund space and $\varphi$ is lower semicontinuous around $\bar{x}$, the $\varepsilon_k$-dependence in the definition of $\partial \varphi(\bar{x}; \bar{u})$ can be omitted, similarly to the non-directional case in \cite[Theorem~2.34]{Mordukhovich2006variational}. 
We next introduce directional metric subregularity, which serves as our primary constraint qualification.

\begin{definition}\label{defn:directional_metric_subregularity}
    Let $F:X\rightrightarrows Y$ be a set-valued mapping between Banach spaces and $(\bar{x},\bar{y})\in\mathrm{gph}\,F$. We say that $F$ is \emph{metrically subregular at $(\bar{x},\bar{y})$ in the direction $u\in X$} whenever the usual error-bound estimate holds throughout some directional neighborhood: there exist $\varepsilon,\delta,\kappa>0$ such that
    \begin{equation} \label{eq:directional_ms}
        \operatorname{dist}\big(x, F^{-1}(\bar{y})\big) \leq \kappa \, \operatorname{dist}\big(\bar{y}, F(x)\big), \quad \forall x \in \mathcal{V}_{\varepsilon,\delta}(\bar{x}; u).
    \end{equation}
For $u=0$, this condition coincides with ordinary metric subregularity. 
\end{definition}

Next, we introduce directional Lipschitz properties for both single-valued and set-valued mappings. 

\begin{definition}[Directional Lipschitz continuity] \label{def:directional_lipschitz_single}
    A single-valued mapping $\varphi:X\to Y$ is \emph{directionally Lipschitz continuous} around $\bar{x}$ along $u$ if its restriction to some directional neighborhood $\mathcal{V}(\bar{x};u)$ is Lipschitz; that is, for some $L\geq0$,
    \begin{equation*}
        \|\varphi(x) - \varphi(z)\| \leq L \|x - z\|, \quad \forall\, x, z \in \mathcal{V}(\bar{x}; u).
    \end{equation*}
\end{definition}

\begin{definition}[Directional Lipschitz-like property] \label{def:directional_lipschitz_set}
    Let $F: X \rightrightarrows Y$ be a set-valued mapping with $\mathrm{dom}\, F \neq \emptyset$, and let $(\bar{x}, \bar{y}) \in \mathrm{gph}\, F$. Let $u \in X$ be a given direction. $F$ is said to be \emph{locally directionally Lipschitz-like} at $(\bar{x}, \bar{y})$ in the direction $u$ if there exist a neighborhood $V$ of $\bar{y}$, a constant $\ell \ge 0$, and a directional neighborhood $\mathcal{V}(\bar{x}; u)$ such that the following inclusion holds:
    \begin{equation} \label{eq:directional_lipschitz_inclusion}
        F(x) \cap V \subset F(z) + \ell \|x - z\| \mathbb{B}_Y, \quad \forall\, x, z \in \mathcal{V}(\bar{x}; u). 
    \end{equation}
\end{definition}

\begin{proposition} \label{prop:lipschitz_implication}
    Let $P: X \to Y$ and let $\Lambda \subset Y$ be a nonempty closed set. Define the set-valued mapping $F: X \rightrightarrows Y$ by $F(x) := P(x) - \Lambda$. Let $\bar{x} \in X$ and $u \in X$ be a given direction. If $P$ is directionally Lipschitz continuous around $\bar{x}$ in the direction $u$, then for any $\bar{y} \in F(\bar{x})$, the mapping $F$ is locally directionally Lipschitz-like at $(\bar{x}, \bar{y})$ in the direction $u$.
\end{proposition}

\begin{proof}
By definition, there exist a constant $L \ge 0$ and a directional neighborhood $\mathcal{V}(\bar{x}; u)$ such that
    \begin{equation} \label{eq:P_lipschitz}
        \|P(x) - P(z)\| \le L \|x - z\| \quad \text{for all } x, z \in \mathcal{V}(\bar{x}; u).
    \end{equation}
    
    Take any $x, z \in \mathcal{V}(\bar{x}; u)$ and choose an arbitrary $y \in F(x)$. By the definition of $F$, there exists $\lambda \in \Lambda$ such that $y = P(x) - \lambda$. Observe that $P(z) - \lambda \in P(z) - \Lambda = F(z)$. By adding and subtracting $P(z)$, we can express $y$ as:
    \begin{equation*}
        y = P(z) - \lambda + (P(x) - P(z)) \in F(z) + (P(x) - P(z)).
    \end{equation*}
    Utilizing the norm bound from \eqref{eq:P_lipschitz}, this directly yields:
    \begin{equation*}
        y \in F(z) + \|P(x) - P(z)\| \mathbb{B}_Y \subset F(z) + L \|x - z\| \mathbb{B}_Y.
    \end{equation*}
    
    Since $y \in F(x)$ was chosen arbitrarily, we obtain the global inclusion:
    \begin{equation*}
        F(x) \subset F(z) + L \|x - z\| \mathbb{B}_Y \quad \text{for all } x, z \in \mathcal{V}(\bar{x}; u).
    \end{equation*}
    
    This global inclusion remains valid when the left-hand side is intersected with any open neighborhood $V$ of $\bar{y}$. Thus, $F$ is directionally Lipschitz-like at $(\bar{x}, \bar{y})$ in the direction $u$ with modulus $\ell := L$. 
\end{proof}

\section{Directional Optimality Conditions}\label{sec:optimality_condition}

In this section, we establish necessary optimality conditions for the optimization problem \eqref{eq:basic_model}. We first present a foundational optimality result in the framework of Asplund spaces (Theorem~\ref{thm:preliminary_optimality_condition}), utilizing the directional coderivative under the assumption of directional metric subregularity. To render these abstract inclusions more applicable to functional constraints, we employ a decoupling lemma (Lemma~\ref{lem:normal_cone_G_representation}) and a directional scalarization theorem (Theorem~\ref{thm:scalarization}). These tools allow us to transform the geometric coderivative of the constraint mapping into the subdifferential of a scalarized functional, leading to a refined optimality condition (Theorem~\ref{thm:further_optimality_mscq}). 

\begin{definition}[Critical Direction] \label{defn:critical_direction} \cite[Definition~5]{gfrerer2013directional}
Let $\bar{x}$ be a feasible solution for \eqref{eq:basic_model}. A vector $u \in X$ is called a \emph{critical direction} of \eqref{eq:basic_model} at $\bar{x}$ if there exist sequences $t_k \downarrow 0$ and $u_k \to u$ such that 
\begin{equation} \label{eq:critical_direction_cond}
    \limsup_{k \to \infty} \frac{f(\bar{x} + t_k u_k) - f(\bar{x})}{t_k} \leq 0 \quad \text{and} \quad \lim_{k \to \infty} \frac{\operatorname{dist}\big(P(\bar{x} + t_k u_k), \Lambda\big)}{t_k} = 0.
\end{equation}
\end{definition}

\begin{remark}\label{rem:critical_direction_equivalence}
       The condition $\lim_{k \to \infty} t_k^{-1} \operatorname{dist}\big(P(\bar{x} + t_k u_k), \Lambda\big) = 0$ is closely tied to the graphical derivative condition. Specifically, if the limit of the difference quotients admits a vector $v$, it implies $v \in DP(\bar{x})(u) \cap T(P(\bar{x}); \Lambda)$, meaning $u$ belongs to the linearized cone $\mathbb{L}(\bar{x}) := \{ u \in X \mid DP(\bar{x})(u) \cap T(P(\bar{x}); \Lambda) \neq \emptyset \}$. 

      Since $f$ is assumed to be locally Lipschitzian, the first condition can be rewritten in terms of the upper Dini directional derivative as
\begin{equation*}
    f_+^\prime(\bar{x}; u) = \limsup_{k \to \infty} \frac{f(\bar{x} + t_k u) - f(\bar{x})}{t_k} \le 0.
\end{equation*}
\end{remark}

Next, we have the following directional necessary optimality condition.

\begin{theorem}\cite[Theorem~7~(ii)]{gfrerer2013directional}\label{thm:preliminary_optimality_condition}
Let $X$ and $Y$ be Asplund spaces. Let $\bar{x}$ be a local minimizer of \eqref{eq:basic_model}, and let $u \in X$ be a critical direction at $\bar{x}$. Assume that $G$ is metrically subregular in the direction $u$ at $(\bar{x}, 0)$. Then, there exists a multiplier $z^* \in Y^*$ such that
\begin{equation} \label{eq:m_stationarity}
    0 \in \partial f(\bar{x}; u) + D_N^* G\big((\bar{x}, 0); (u, 0)\big)(z^*).
\end{equation}
\end{theorem}

By requiring metric subregularity only along a specific direction $u$, Theorem~\ref{thm:preliminary_optimality_condition} accommodates problems where standard constraint qualifications fail. However, the directional limiting coderivative appearing in \eqref{eq:m_stationarity} is difficult to calculate in practice, which motivates a deeper investigation into its explicit structure. 

To clarify the structure of $D_N^* G\big((\bar{x}, 0); (u, 0)\big)(z^*)$ and to facilitate the verification of the directional metric subregularity in Theorem~\ref{thm:preliminary_optimality_condition}, we establish the following decoupling lemma. This result extends the finite-dimensional decoupling relation found in \cite[Lemma~3.2]{bai2019directional} to Asplund spaces.

\begin{lemma} \label{lem:normal_cone_G_representation}
Let $P:X\to Y$ be continuous at $x$, and let $G: X \rightrightarrows Y$ be defined by $G(x) := P(x) - \Lambda$. Let $(x, y) \in \mathrm{gph}\, G$ and $s \in \Lambda$ satisfy $y = P(x) - s$. If $(x^*, -y^*) \in \widehat{N}\left((x, y); \mathrm{gph}\, G\right)$, then
\begin{equation} \label{eq:decoupling_inclusion}
    x^* \in \widehat{D}_N^* P(x)(y^*) \quad \text{and} \quad y^* \in \widehat{N}(s; \Lambda).
\end{equation}
\end{lemma}

\begin{proof}
Fix $\varepsilon>0$. By the definition of the Fr\'echet normal cone, there exists a neighborhood $U$ of $(x,y)$ such that
\begin{equation} \label{eq:frechet_normal_G_def}
    \langle x^*, x' - x \rangle + \langle -y^*, y' - y \rangle \leq \varepsilon \|(x' - x, y' - y)\|, \quad \forall (x', y') \in \mathrm{gph}\, G \cap U.
\end{equation}
To prove the second inclusion in \eqref{eq:decoupling_inclusion}, we choose $x' = x$ and $y' = P(x) - s'$ for any $s' \in \Lambda$. When $\|s' - s\|$ is sufficiently small, $(x, y') \in \mathrm{gph}\, G \cap U$. Substituting these into \eqref{eq:frechet_normal_G_def}, we obtain:
\begin{equation*}
    \langle -y^*, (P(x) - s') - (P(x) - s) \rangle \leq \varepsilon \|s' - s\| \iff \langle y^*, s' - s \rangle \leq \varepsilon \|s' - s\|,
\end{equation*}
which implies $y^* \in \widehat{N}(s; \Lambda) = \widehat{N}(P(x) - y; \Lambda)$.

Next, to establish the first inclusion, take an arbitrary $x' \in X$ near $x$ and set $y' := P(x') - s$. Since $s \in \Lambda$, we have $y' \in G(x')$ and thus $(x', y') \in \mathrm{gph}\, G$. Applying \eqref{eq:frechet_normal_G_def} again yields:
\begin{equation*}
    \langle x^*, x' - x \rangle + \langle -y^*, P(x') - P(x) \rangle \leq \varepsilon \|(x' - x, P(x') - P(x))\|.
\end{equation*}
By definition, this means that $(x^*, -y^*) \in \widehat{N}((x, P(x)); \mathrm{gph}\, P)$, which is equivalent to $x^* \in \widehat{D}_N^* P(x)(y^*)$.
\end{proof}

Lemma~\ref{lem:normal_cone_G_representation} provides a useful necessary decoupling relation for the Fr\'echet normal cone to the graph of $G$. It shows that any normal pair $(x^*, -y^*)$ to $\mathrm{gph}\, G$ induces a coderivative element of $P$ and a normal to the constraint set $\Lambda$. To further describe coderivatives via scalarization, we introduce the following concept. 

\begin{definition}[Weak\(^*\) strict Lipschitz property {\cite[Definition~3.25]{Mordukhovich2006variational}}] \label{defn:weak_strict_lip}
Let $\phi:X\to Y$ be Lipschitz continuous around $\bar{x}$. For $u\in X$, $x_k\to\bar{x}$, and $t_k\downarrow0$, form the difference quotients $q_k:=\frac{\phi(x_k+t_ku)-\phi(x_k)}{t_k}$. 
The mapping $\phi$ is called \emph{weak\(^*\) strictly Lipschitzian} at $\bar{x}$ if $\langle y_k^*,q_k\rangle\to0$ for every such choice of $u$, $\{x_k\}$, and $\{t_k\}$ and every sequence $y_k^*\xrightarrow{w^*}0$ in $Y^*$.
\end{definition}

Notably, this class of mappings includes Fredholm integral operators with Lipschitzian kernels, which are particularly important in applications to optimal control \cite[p.~288]{Mordukhovich2006variational}. In finite-dimensional settings (i.e., $\dim Y < \infty$), it is automatically satisfied and reduces to classical local Lipschitz continuity. 

Next, we present a directional scalarization formula for the coderivative $D_N^* \phi$. When $\phi$ is locally Lipschitz, this formula generalizes the classical nondirectional counterpart established in \cite[Theorem~3.28]{Mordukhovich2006variational}. 

\begin{theorem} \label{thm:scalarization}
Let $X,Y$ be Asplund spaces. Suppose that the mapping $\phi: X \to Y$ is locally Lipschitz and weak\(^*\) strictly Lipschitzian at $\bar{x} \in X$. If $\phi$ is Hadamard directionally differentiable at $\bar{x}$ in a given direction $\bar{u} \in X$, then, for any $y^* \in Y^*$, we have
\begin{equation} \label{eq:scalarization_formula}
    D_N^* \phi(\bar{x}; \bar{u})(y^*) = \partial \langle y^*, \phi \rangle (\bar{x}; \bar{u}).
\end{equation}
\end{theorem}

The proof is detailed in \cite[Theorem 3.10]{mao2025directional}. 
Theorem~\ref{thm:scalarization} establishes an exact directional scalarization formula for the basic coderivative in Asplund spaces. This precise relationship simplifies the characterization of $D_N^* \phi$ by facilitating the scalar subdifferential calculus. 

\begin{remark} \label{rem:optimality_restated} 
    Let $z^* \in Y^*$ be the multiplier obtained in \eqref{eq:m_stationarity}, and consider $(x^*, -z^*) \in N\big((\bar{x}, 0); \mathrm{gph}\, G; (u, 0)\big)$, with $P$ Hadamard directionally differentiable at $\bar{x}$ in the direction $u$. By the definition of the directional limiting normal cone, there exist sequences $t_k \downarrow 0$, $u_k \to u$, $v_k \to 0$ with $y_k := t_k v_k \in G(x_k)$, and dual sequences $x_k^* \xrightarrow{w^*} x^*$, $z_k^* \xrightarrow{w^*} z^*$ such that
    \begin{equation*}
        (x_k^*, -z_k^*) \in \widehat{N}\left((x_k, y_k); \mathrm{gph}\, G\right), \quad \text{where } x_k := \bar{x} + t_k u_k.
    \end{equation*}
    Since $y_k \in G(x_k) = P(x_k) - \Lambda$, there exists a sequence $s_k \in \Lambda$ such that $y_k = P(x_k) - s_k$. Applying Lemma~\ref{lem:normal_cone_G_representation} to the normal inclusion above gives
    \begin{equation} \label{eq:frechet_decoupling}
        x_k^* \in \widehat{D}_N^* P(x_k)(z_k^*) \quad \text{and} \quad z_k^* \in \widehat{N}(s_k; \Lambda), \quad \text{with } s_k = P(x_k) - y_k.
    \end{equation}
    Observe the asymptotic behavior of $s_k$:
    \begin{equation*}
        \frac{s_k - P(\bar{x})}{t_k} = \frac{P(x_k) - P(\bar{x})}{t_k} - v_k \to P'_H(\bar{x}; u).
    \end{equation*}
    Taking the limit as $k \to \infty$, the sequential closedness of the directional limiting normal cone applied to the second inclusion in \eqref{eq:frechet_decoupling} yields $z^* \in N\big(P(\bar{x}); \Lambda; P'_H(\bar{x}; u)\big)$. Concurrently, the first inclusion in \eqref{eq:frechet_decoupling} passes to the directional limiting coderivative $x^* \in D_N^* P\big(\bar{x}; (u, P'_H(\bar{x}; u))\big)(z^*)$. 
\end{remark}

Next, we present the following theorem as an extension of \cite[Theorem~7]{gfrerer2013directional} in Asplund spaces. 

\begin{theorem} \label{thm:further_optimality_mscq}
    Let $X$ and $Y$ be Asplund spaces. Let $\bar{x}$ be a local minimizer of \eqref{eq:basic_model} and $\bar{u} \in X$ be a critical direction at $\bar{x}$. Assume that $G$ is metrically subregular in the direction $\bar{u}$ at $(\bar{x}, 0)$. Furthermore, assume that $P$ is weak\(^*\) strictly Lipschitzian at $\bar{x}$ and Hadamard directionally differentiable at $\bar{x}$ in the direction $\bar{u}$ with $\bar{v} := P'_H(\bar{x}; \bar{u}) \in T(P(\bar{x}); \Lambda)$. Then, there exists a multiplier $z^* \in N\big(P(\bar{x}); \Lambda; \bar{v}\big)$ such that
    \begin{equation} \label{eq:refined_optimality}
        0 \in \partial f(\bar{x}; \bar{u}) + D_N^* P\big(\bar{x}; (\bar{u}, \bar{v})\big)(z^*) = \partial f(\bar{x}; \bar{u}) + \partial \langle z^*, P \rangle (\bar{x}; \bar{u}).
    \end{equation}
\end{theorem}

\begin{proof}
    Let $\Omega := P^{-1}(\Lambda)$ denote the feasible region. Since $\bar{x}$ is a local minimizer of \eqref{eq:basic_model}, it is a local minimizer of the unconstrained objective function $f + \delta_\Omega$. By the generalized Fermat's rule for directional subdifferentials, we have
    \begin{equation*}
        0 \in \partial (f + \delta_\Omega)(\bar{x}; \bar{u}).
    \end{equation*}
    Since $f$ is locally Lipschitz and the underlying spaces are Asplund, the regularity condition is satisfied (also guaranteed by the directional metric subregularity for the indicator function). We can thus apply the directional sum rule \cite[Theorem 2.13]{mao2025directional} to separate the subdifferential, which yields
    \begin{equation} \label{eq:sum_rule_applied}
        \partial(f + \delta_\Omega)(\bar{x}; \bar{u}) \subset \partial f(\bar{x}; \bar{u}) + \partial \delta_\Omega(\bar{x}; \bar{u}) = \partial f(\bar{x}; \bar{u}) + N(\bar{x}; \Omega; \bar{u}).
    \end{equation}
 
    Applying the directional normal cone representation for preimages \cite[Theorem~2.1]{mao2025directional} under the assumed directional metric subregularity of $G$, we can evaluate the normal cone to the feasible set $\Omega$:
    \begin{equation} \label{eq:chain_rule_preimage}
        N(\bar{x}; \Omega; \bar{u}) \subset D_N^* P\big(\bar{x}; (\bar{u}, \bar{v})\big) \big( N(P(\bar{x}); \Lambda; \bar{v}) \big).
    \end{equation}
    Combining the inclusions \eqref{eq:sum_rule_applied} and \eqref{eq:chain_rule_preimage}, there must exist a multiplier $z^* \in N(P(\bar{x}); \Lambda; \bar{v})$ such that 
    \begin{equation*}
        0 \in \partial f(\bar{x}; \bar{u}) + D_N^* P\big(\bar{x}; (\bar{u}, \bar{v})\big)(z^*). 
    \end{equation*}

    Finally, since $P$ is weak\(^*\) strictly Lipschitzian at $\bar{x}$, \eqref{eq:refined_optimality} follows directly from the scalarization formula established in Theorem~\ref{thm:scalarization}. This allows us to represent the directional limiting coderivative in terms of the subdifferential of the scalarized mapping $\partial \langle z^*, P \rangle (\bar{x}; \bar{u})$, completing the proof.   
\end{proof}

Theorem~\ref{thm:further_optimality_mscq} provides a refined optimality condition by explicitly coupling the variation of $P$ with the direction $\bar{u}$. Unlike non-directional variants, the inclusion of $\bar{v} = P'_H(\bar{x}; \bar{u})$ guarantees that the multiplier $z^*$ is selected from the narrower directional normal cone $N(P(\bar{x}); \Lambda; \bar{v})$, reflecting the fact that only those constraints active along the direction $\bar{v}$ contribute to stationarity.

\section{Sufficient Conditions for Directional Metric Subregularity} \label{sec:sufficient_condition_for_mscq}

Up to now, we have characterized directional optimality conditions for \eqref{eq:basic_model} under the assumption of directional metric subregularity. However, in many practical scenarios, verifying directional metric subregularity directly can be quite challenging. In this section, we aim to establish several sufficient conditions for directional metric subregularity. 

\begin{definition}[Directional Pseudo-normality] \label{defn:directional_pseudo_normality} 
    Suppose that $P(\bar{x}) \in \Lambda$. Directional pseudo-normality is said to hold at $\bar{x}$ in a given direction $u$ if for every $v \in DP(\bar{x})(u) \cap T(P(\bar{x}); \Lambda)$,  there exists no multiplier $y^* \in Y^* \setminus \{0\}$ such that the following three conditions hold simultaneously:
    \begin{enumerate}
        \item[(i)] $0 \in D_N^* P\big(\bar{x}; (u,v)\big)(y^*)$ and $y^* \in N(P(\bar{x}); \Lambda; v)$;
        \item[(ii)] there exist sequences $t_k \downarrow 0$, $u_k \to u$, $s_k \to P(\bar{x})$, and $y_k^* \xrightarrow{w^*} y^*$ satisfying $y_k^* \in \widehat{N}(s_k; \Lambda)$;
        \item[(iii)] for all sufficiently large $k$, it holds that
        \begin{equation}\label{eq:directional_pseudo_normality}
            \langle y_k^*, P(\bar{x} + t_k u_k) - s_k \rangle > 0.  
        \end{equation}
    \end{enumerate}
\end{definition}

\begin{remark}
    Our definition of directional pseudo-normality differs from its finite-dimensional counterpart. When $Y$ is finite-dimensional, the strict inequality in \eqref{eq:directional_pseudo_normality} is conventionally formulated using the limit multiplier $y^*$ rather than the sequence $y_k^*$ \cite[Definition~4.1]{bai2019directional}, because weak\(^*\) convergence and norm convergence coincide. In infinite-dimensional spaces, weak\(^*\) convergence alone does not preserve a strict inequality involving a varying primal sequence. We therefore formulate the condition directly with $y_k^*$. No general implication between this sequential sign condition and the corresponding condition written with $y^*$ should be inferred without additional assumptions.
\end{remark}

When the image space is finite-dimensional, we record the corresponding notion, which follows \cite[Definition~4.1]{bai2019directional}.

\begin{definition}[Directional Quasi-Normality] \label{defn:directional_quasi_normality}
    Suppose that $P(\bar{x}) \in \Lambda$ and that $Y=\mathbb{R}^d$. Directional quasi-normality is said to hold at $\bar{x}$ in a given direction $u$ if for every $v \in DP(\bar{x})(u) \cap T(P(\bar{x}); \Lambda)$, there exists no multiplier $y^* \in Y^* \setminus \{0\}$ such that the following three conditions hold simultaneously:
    \begin{enumerate}
        \item[(i)] $0 \in D_N^* P\big(\bar{x}; (u,v)\big)(y^*)$ and $y^* \in N(P(\bar{x}); \Lambda; v)$;
        \item[(ii)] there exist sequences $t_k \downarrow 0$, $u_k \to u$, $s_k \to P(\bar{x})$, and $y_k^* \to y^*$ satisfying $y_k^* \in \widehat{N}(s_k; \Lambda)$;
        \item[(iii)] for all sufficiently large $k$, it holds that
        \begin{equation}\label{eq:directional_quasi_normality}
            [y^*]_i[P(\bar{x} + t_k u_k) - s_k]_i> 0,\; i=1,\dots,d, \text{ if } [y^*]_i\neq 0.
        \end{equation}
    \end{enumerate}
\end{definition}

The following lemma gives a sequential sufficient condition for directional metric subregularity. It extends the contradiction criterion in \cite[Lemma~3.1]{bai2019directional}.

\begin{lemma} \label{lem:sequential_ms_characterization}
Let $G: X \rightrightarrows Y$ be a set-valued mapping with a closed graph between Asplund spaces $X$ and $Y$, and let $(\bar{x}, \bar{y}) \in \mathrm{gph}\, G$. Assume that for a given direction $u \in X$, there do not exist sequences $t_k \downarrow 0$, $\|(u_k, v_k)\| = 1$, and $\|y_k^*\| = 1$ with $\|u_k\| \to 1$, $\|u\|u_k \to u$, $v_k \to 0$, and $x_k^* \to 0$ such that
\begin{equation} \label{eq:sequential_ms_condition}
    (x_k^*, -y_k^*) \in \widehat{N}\big((x_k', y_k'); \mathrm{gph}\, G\big), \quad x_k' \notin G^{-1}(\bar{y}), \quad \lim_{k\to \infty} \frac{\langle y_k^*, y_k' - \bar{y} \rangle}{\|y_k' - \bar{y}\|} = 1,
\end{equation}
where $x_k' := \bar{x} + t_k u_k \neq \bar{x}$ and $y_k' := \bar{y} + t_k v_k \neq \bar{y}$. Then $G$ is metrically subregular at $(\bar{x}, \bar{y})$ in the direction $u$.
\end{lemma}

\begin{proof}
Since $X$ and $Y$ are Asplund spaces, the conclusion follows directly from \cite[Corollary~1, Remark 2]{gfrerer2014metric}. Directional metric subregularity is positively homogeneous in the direction, so if $u\neq 0$ one may assume $\|u\|=1$ without loss of generality; in that case the condition $\|u\| u_k \to u$ reduces to $u_k \to u$.  
\end{proof}

Lemma~\ref{lem:sequential_ms_characterization} translates the abstract property of metric subregularity into a sequence-based contradiction argument. Still, there is a fundamental challenge in infinite-dimensional variational analysis: the difference between the weak\(^*\) topology and the norm topology, which can result in the ``vanishing'' of multipliers (i.e., $\|y_k^*\|=1$ but $y_k^*\xrightarrow{w^*} 0$). The following example illustrates that this phenomenon can occur even when $Y$ is a Hilbert space. 

\begin{example}
Consider the Hilbert space $Y=\ell^2$. Let $\{e_n\}_{n \in \mathbb{N}}$ be the sequence of standard unit vectors in $\ell^2$. Clearly, $\|e_n\|_2 = 1$ for every $n$. However, for any $x = (x_1, x_2, \dots) \in \ell^2$, we have $\langle e_n, x \rangle = x_n \to 0$ as $n \to \infty$. Thus, $e_n \xrightarrow{w} 0$, demonstrating that in infinite-dimensional settings, a sequence can be weakly convergent to zero even if its norm is a nonzero constant. 
\end{example}

To overcome this issue, the following condition serves as a recovery mechanism: for the specified normal sequences, weak\(^*\) convergence of the multipliers to zero, together with strong convergence of the \(X^*\)-components to zero, forces the multiplier norms to vanish.

\begin{definition}[Directional PSNC] \label{def:directionally_psnc}
Let $M: X \rightrightarrows Y$ be a set-valued mapping between Asplund spaces, and let $(\bar{x}, \bar{y}) \in \mathrm{gph}\, M$. Given a direction $(u, v) \in X \times Y$, the mapping $M$ is said to be \emph{partially sequentially normally compact (PSNC)} in the direction $(u, v)$ at $(\bar{x}, \bar{y})$ if, for any sequences $t_k \downarrow 0$, $(u_k, v_k) \to (u, v)$, and any dual sequences $(x_k^*, y_k^*) \in X^* \times Y^*$ satisfying 
\begin{equation*}
    (x_k^*, -y_k^*) \in \widehat{N} \big( (\bar{x} + t_k u_k, \bar{y} + t_k v_k); \mathrm{gph}\, M \big),
\end{equation*}
the conditions $x_k^* \to 0$ and $y_k^* \xrightarrow{w^*} 0$ imply that $\lim_{k \to \infty} \|y_k^*\| = 0$. 
\end{definition}

It is well known that the Lipschitz-like property implies PSNC \cite[Theorem~1.44]{Mordukhovich2006variational}. We now present the Asplund space version of the results in \cite[Theorem~3.1]{bai2019directional}. 

\begin{definition}[Weak Sufficient Condition for Metric Subregularity] \label{defn:wscms}
    Let $P(\bar{x}) \in \Lambda$ and let $u \in X$ with $\|u\|=1$. Assume that $P$ is Hadamard directionally differentiable at $\bar{x}$ in the direction $u$ with $v := P'_H(\bar{x}; u)$. The \emph{weak sufficient condition for metric subregularity} (WSCMS) is said to hold at $\bar{x}$ in the direction $u$ if there exists no multiplier $y^* \in Y^* \setminus \{0\}$, no sequences $t_k \downarrow 0$, and $u_k \to u$, $w_k \to 0$ ($w_k \neq 0$), $y_k^* \xrightarrow{w^*} y^*$ satisfying the following conditions simultaneously:
    \begin{enumerate}
        \item[\rm(1)] $0 \in D_N^* P\big(\bar{x}; (u,v)\big)(y^*)$ and $y^* \in N(P(\bar{x}); \Lambda; v)$;
        \item[\rm(2)] $y_k^* \in \widehat{N}(s_k; \Lambda)$, where $s_k = P(\bar{x} + t_k u_k) - t_k w_k$ and $P(\bar{x} + t_k u_k) \notin \Lambda$;
        \item[\rm(3)] $\lim_{k \to \infty} \left\langle y_k^*, \frac{w_k}{\|w_k\|} \right\rangle = 1$.
    \end{enumerate}
\end{definition}

\begin{theorem} \label{thm:sufficient_condition_for_mscq}
    Let $P(\bar{x})\in\Lambda$, and let $u \in X$ with $\|u\| = 1$ be a given direction. Assume that $P$ is Hadamard directionally differentiable at $\bar{x}$ in the direction $u$ with $v := P'_H(\bar{x}; u)$, and that $P$ is directionally PSNC at $(\bar{x}, P(\bar{x}))$ in the direction $(u, v)$. If WSCMS holds at $\bar{x}$ in the direction $u$, then $G(x) := P(x) - \Lambda$ is metrically subregular at $(\bar{x}, 0)$ in the direction $u$.
\end{theorem}

\begin{proof}
    Assume to the contrary that $G(x) = P(x) - \Lambda$ is not metrically subregular at $(\bar{x}, 0)$ in the direction $u$. By the sequential characterization of directional metric subregularity in Lemma~\ref{lem:sequential_ms_characterization}, together with the decoupling in Lemma~\ref{lem:normal_cone_G_representation}, there exist sequences $t_k \downarrow 0$, $w_k \to 0$ (serving as the residual direction), normalized multiplier sequences $\|y_k^*\| = 1$, and directions $u_k \to u$, $x_k^* \to 0$ such that:
    \begin{equation} \label{eq:proof_sequences}
        (x_k^*, -y_k^*) \in \widehat{N}\bigl((\bar{x} + t_k u_k, P(\bar{x} + t_k u_k)); \mathrm{gph}\, P\bigr), \quad y_k^* \in \widehat{N}\bigl(s_k; \Lambda\bigr),
    \end{equation}
    where $s_k := P(\bar{x} + t_k u_k) - t_k w_k \in \Lambda$, and $\lim_{k \to \infty} \langle y_k^*, \frac{w_k}{\|w_k\|} \rangle = 1$. 

    Since $Y$ is Asplund, after passing to a subsequence, we may suppose that $y_k^* \xrightarrow{w^*} y^*$ in $Y^*$ with $\|y^*\| \leq 1$.
    
    Next, we explicitly show that $v \in T(P(\bar{x}); \Lambda)$. By the definition of the normal cone in \eqref{eq:proof_sequences}, we must have $s_k \in \Lambda$ for all $k$. We can rewrite $s_k$ as $s_k = P(\bar{x}) + t_k \tilde{w}_k - t_k w_k$, where $\tilde{w}_k := \frac{P(\bar{x} + t_k u_k) - P(\bar{x})}{t_k}$. By the Hadamard directional differentiability of $P$ at $\bar{x}$ in direction $u$, and since $u_k \to u$, we have $\tilde{w}_k \to v$. Because $w_k \to 0$, it follows that $\frac{s_k - P(\bar{x})}{t_k} = \tilde{w}_k - w_k \to v - 0 = v$.
    
    Since $s_k \in \Lambda$ and $P(\bar{x}) \in \Lambda$, the sequential limit definition of the tangent cone immediately yields $v \in T(P(\bar{x}); \Lambda)$. This ensures $v \in DP(\bar{x})(u) \cap T(P(\bar{x}); \Lambda)$. 

    Returning to \eqref{eq:proof_sequences}, taking the sequential limit as $k \to \infty$, the definition of the directional limiting normal cone combined with $(u_k, \tilde{w}_k) \to (u, v)$ yields:
    \begin{equation*}
        (0, -y^*) \in N\bigl((\bar{x}, P(\bar{x})); \mathrm{gph}\, P; (u, v)\bigr) \quad \text{and} \quad y^* \in N(P(\bar{x}); \Lambda; v).
    \end{equation*}
    This establishes condition (1).

    If $y^* = 0$, the directional PSNC property of $P$ at $(\bar{x}, P(\bar{x}))$ in the direction $(u,v)$ asserts that the combination of $y_k^* \xrightarrow{w^*} 0$, $x_k^* \to 0$, and $(u_k, \tilde{w}_k) \to (u, v)$ mandates the strong convergence $\lim_{k \to \infty} \|y_k^*\| = 0$. This contradicts the normalization condition $\|y_k^*\| = 1$. Therefore, it must hold that $y^* \neq 0$.

    We have constructed a non-zero multiplier $y^*$ and sequences satisfying conditions (1), (2), and (3) of the WSCMS simultaneously. This contradicts our initial premise and completes the proof.  
\end{proof}

Theorem~\ref{thm:sufficient_condition_for_mscq} establishes a sequential verification criterion for directional metric subregularity in Asplund spaces by leveraging the directional PSNC property. By ruling out the existence of pathological directions and multipliers that satisfy the specified directional condition $\lim_{k \to \infty} \langle y_k^*, w_k/\|w_k\| \rangle = 1$, this result guarantees a sharp local error bound for the mapping $G(x) = P(x) - \Lambda$. 

Moreover, \cite[Corollary~1]{gfrerer2014metric} established a related sequential criterion by employing a sequence of $\varepsilon$-regular normal cones $\widehat{N}_{\varepsilon_k}$ evaluated at perturbed points $(\bar{x} + t_k u_k, \bar{y} + t_k v_k)$ with $u_k \to u$ and $v_k \to 0$. In contrast, Theorem~\ref{thm:sufficient_condition_for_mscq} successfully generalizes directional metric subregularity to composite structures of the form $P(x) - \Lambda$ by leveraging Hadamard differentiability and PSNC. 

To further simplify the validation of the WSCMS condition in Definition~\ref{defn:wscms}, we establish a technical tool linking the alignment of dual multipliers with the functional variation of the mapping $P$. The following corollary demonstrates that the alignment condition in WSCMS naturally triggers the strict inequality required in the definition of directional pseudo-normality, thereby generalizing \cite[Corollary~4.1]{bai2019directional} to Asplund spaces. 

\begin{corollary} \label{cor:dpn_to_mscq}
    Suppose that $P$ is Hadamard directionally differentiable at $\bar{x}$ in a given direction $u \in X$ with $v:=P'_H(\bar{x};u)$, and that $P$ is directionally PSNC at $(\bar{x},P(\bar{x}))$ in the direction $(u,v)$. If directional pseudo-normality holds at $\bar{x}$ in the direction $u$, then $G(x) := P(x) - \Lambda$ is metrically subregular at $(\bar{x}, 0)$ in the direction $u$. 
\end{corollary}

\begin{proof}
    Assume to the contrary that $G$ is not metrically subregular at $(\bar{x}, 0)$ in the direction $u$. By the sequential characterization (Lemma~\ref{lem:sequential_ms_characterization}), there exist sequences $t_k \downarrow 0$, $w_k \to 0$ ($w_k \neq 0$), and normalized multipliers $\|y_k^*\| = 1$ with $u_k \to u$ and $x_k^* \to 0$ such that
    \begin{equation} \label{eq:cor_proof_inclusion}
        (x_k^*, -y_k^*) \in \widehat{N}\big((x_k', y_k'); \mathrm{gph}\, G\big), \quad \lim_{k\to \infty} \left\langle y_k^*, \frac{y_k'}{\|y_k'\|} \right\rangle = 1,
    \end{equation}
    where $x_k' := \bar{x} + t_k u_k$ and $y_k' := t_k w_k \neq 0$. By the definition of $G(x) = P(x) - \Lambda$, $y_k' \in G(x_k')$ implies that there exists $s_k \in \Lambda$ such that $y_k' = P(x_k') - s_k$.

    Applying the decoupling lemma (Lemma~\ref{lem:normal_cone_G_representation}) to the Fr\'echet normal cone inclusion in \eqref{eq:cor_proof_inclusion}, we obtain:
    \begin{equation} \label{eq:cor_decoupled}
        x_k^* \in \widehat{D}_N^* P(x_k')(y_k^*) \quad \text{and} \quad y_k^* \in \widehat{N}(s_k; \Lambda).
    \end{equation}
    Since $Y$ is Asplund, after passing to a subsequence, we may suppose that $y_k^* \xrightarrow{w^*} y^*$ in the closed unit ball of $Y^*$. Because $x_k^* \to 0$, if $y^* = 0$, the directional PSNC property of $P$ would force $\lim_{k \to \infty} \|y_k^*\| = 0$, which contradicts $\|y_k^*\| = 1$. Thus, we must have $y^* \neq 0$. 

    Furthermore, since $P$ is Hadamard directionally differentiable at $\bar{x}$ in the direction $u$, we define $v := P'_H(\bar{x}; u)$. It follows that $\frac{P(\bar{x} + t_k u_k) - P(\bar{x})}{t_k} \to v$. Consequently, taking the sequential limit as $k \to \infty$ in \eqref{eq:cor_decoupled} gives:
    \begin{equation*}
        0 \in D_N^* P\big(\bar{x}; (u, v)\big)(y^*) \quad \text{and} \quad y^* \in N(P(\bar{x}); \Lambda; v),
    \end{equation*}
    which also implies $v \in DP(\bar{x})(u) \cap T(P(\bar{x}); \Lambda)$ by definition.

    Finally, we translate the alignment condition from \eqref{eq:cor_proof_inclusion}. Since $y_k' = t_k w_k$, the condition $\lim_{k \to \infty} \langle y_k^*, w_k/\|w_k\| \rangle = 1$ ensures that for all sufficiently large $k$, we directly have $\langle y_k^*, w_k/\|w_k\| \rangle > 0$. Multiplying this inequality by the positive scalar $t_k \|w_k\|$, we obtain:
    \begin{equation*}
        \langle y_k^*, t_k w_k \rangle = \langle y_k^*, P(\bar{x} + t_k u_k) - s_k \rangle > 0.
    \end{equation*}

    We have therefore found a multiplier $y^* \neq 0$, along with corresponding sequences, that simultaneously satisfy all the conditions of Definition~\ref{defn:directional_pseudo_normality}. This contradicts the assumption that directional pseudo-normality holds at $\bar{x}$. Thus, $G$ must be metrically subregular at $(\bar{x}, 0)$ in the direction $u$.  
\end{proof}

When the image space $Y = \mathbb{R}^d$ is finite-dimensional, the bounded sequence of multipliers $\{y_k^*\}$ on the unit sphere admits a strongly convergent subsequence. Consequently, the limit multiplier inherently satisfies $\|y^*\| = 1 \neq 0$, rendering the PSNC assumption automatically fulfilled. However, the Hadamard directional differentiability of $P$ remains a crucial structural requirement. Under this setting, the metric subregularity can be secured via the directional quasi-normality of Definition~\ref{defn:directional_quasi_normality}.

\begin{corollary} \label{cor:dqn_to_mscq}
    Let $Y = \mathbb{R}^d$. Suppose that $P$ is Hadamard directionally differentiable at $\bar{x}$ in a given direction $u \in X$. If the directional quasi-normality of Definition~\ref{defn:directional_quasi_normality} holds at $\bar{x}$ in the direction $u$, then $G(x) := P(x) - \Lambda$ is metrically subregular at $(\bar{x}, 0)$ in the direction $u$. 
\end{corollary}

\begin{proof}
The argument is the same as that of Corollary~\ref{cor:dpn_to_mscq}, except that compactness of the unit sphere in $\mathbb{R}^d$ yields a strongly convergent subsequence of $\{y_k^*\}$ with $\|y^*\|=1$. Directional PSNC is therefore automatic. In finite dimensions the pairing condition of directional quasi-normality is imposed on the limit multiplier $y^*$, which is legitimate because the topologies of $Y^*$ coincide.  
\end{proof}

Consider an optimization problem in Asplund spaces:
\begin{equation} \label{eq:ocpec}
    \min_{x \in X} \quad J(x) \quad \text{s.t. } P_1(x) \in \Lambda_1, \quad P_2(x) \in \Lambda_2,
\end{equation}
where $J: X \to \mathbb{R}$ is a locally Lipschitz function, $P_1: X \to Y_1$ and $P_2: X \to Y_2$ are continuous mappings between Asplund spaces, and $\Lambda_1 \subseteq Y_1, \Lambda_2 \subseteq Y_2$ are closed sets. 
In practice, $\mathcal{F}$ frequently involves coupled mappings, such as the value function reformulations of bilevel problems~\cite{ye1995necessary,Ye2010}.  

Next, consider two set-valued mappings $F_1:=P_1-\Lambda_1: X \rightrightarrows Y_1$ and $F_2:=P_2-\Lambda_2: X \rightrightarrows Y_2$. Define the set-valued mapping $F: X \rightrightarrows Y_1 \times Y_2$ by $F(x) := \big( F_1(x), F_2(x) \big)$. 
The following definition records the sequential normal configuration used for the joint system. The first constraint is the part whose directional metric subregularity is known in advance, whereas the residual of the second constraint is used in the variational argument. 

\begin{definition}[Joint directional pseudo/quasi-normality]\label{defn:joint_directional_pseudo_normality}
Let $F_i:X\rightrightarrows Y_i$ have closed graphs, let $\bar{y}_i\in F_i(\bar{x})$, $i=1,2$, and let $u\in X$ be given.

\begin{enumerate}
\item[(i)] The \emph{joint directional pseudo-normality} condition for $F_1$ and $F_2$ holds at $(\bar{x},\bar{y}_1,\bar{y}_2)$ in the direction $u$ if there do not exist sequences $t_k\downarrow0$, $(x_{i,k},y_{i,k})\in\operatorname{gph}F_i$, and dual elements $(x_{i,k}^*,y_{i,k}^*)\in X^*\times Y_i^*$, $i=1,2$, such that
\begin{equation}\label{eq:sequential_qualification_primal}
    \frac{x_{i,k}-\bar{x}}{t_k}\to u,
    \qquad
    \frac{y_{i,k}-\bar{y}_i}{t_k}\to0,
    \qquad i=1,2,
\end{equation}
\begin{equation}\label{eq:sequential_decoupled_normals}
\begin{aligned}
    &x_{i,k}^*\in \widehat D_N^*P_i(x_{i,k})(y_{i,k}^*),\;\; 
    y_{i,k}^*\in \widehat N(s_{i,k};\Lambda_i),\\
    &s_{i,k}:=P_i(x_{i,k})-y_{i,k}\in\Lambda_i,
    \qquad i=1,2.
\end{aligned}
\end{equation}
and
\begin{equation}\label{eq:sequential_qualification_dual}
    \|x_{1,k}^*+x_{2,k}^*\|\to0,
    \qquad
    \sup_k\|y_{1,k}^*\|<\infty,
    \qquad
    \|y_{2,k}^*\|=1,
\end{equation}
and
\begin{equation}\label{eq:joint_directional_pseudo_normality}
    \langle y_{2,k}^*,y_{2,k}-\bar{y}_2\rangle>0
    \qquad\text{for all sufficiently large }k.
\end{equation}

\item[(ii)] If $Y_2=\mathbb R^d$, the \emph{joint directional quasi-normality} condition for $F_1$ and $F_2$ holds at $(\bar{x},\bar{y}_1,\bar{y}_2)$ in the direction $u$ if there do not exist sequences satisfying \eqref{eq:sequential_qualification_primal}--\eqref{eq:sequential_qualification_dual} and a $y_2^*$ such that $y_{2,k}^*\to y_2^*$, with \eqref{eq:joint_directional_pseudo_normality} replaced by
\begin{equation*}
    [y_2^*]_j[y_{2,k}-\bar{y}_2]_j>0
    \quad\text{for every }j\text{ with }[y_2^*]_j\ne0
    \quad\text{and all sufficiently large }k.
\end{equation*}
\end{enumerate}
\end{definition}

The following theorem addresses directional metric subregularity for the intersection of two closed constraints. 

\begin{theorem}\label{thm:directional_mscq_intersection}
Let $X$, $Y_1$, and $Y_2$ be Asplund spaces, and endow $Y_1\times Y_2$ with the sum norm $\|(y_1,y_2)\|=\|y_1\|+\|y_2\|$. 
Let $(\bar{x}, \bar{y}_1, \bar{y}_2)$ be a reference point satisfying $\bar{y}_1 \in F_1(\bar{x})$ and $\bar{y}_2 \in F_2(\bar{x})$. Suppose that $u \in X$ is a given direction such that \(0\in DF_i(\bar{x}\mid \bar{y}_i)(u)\) for $i=1,2$. Assume that the following conditions hold:

\begin{enumerate}
\item[(i)] $F_1$ is metrically subregular at $(\bar{x},\bar{y}_1)$ in the direction $u$, and $F_2$ is directionally Lipschitz-like at $(\bar{x}, \bar{y}_2)$ in the direction $u$. In addition, $\operatorname{dist}(\bar{y}_2,F_2(x))$ is attained whenever $x$ belongs to a sufficiently small directional neighborhood of $\bar{x}$ along $u$.

\item[(ii)] Joint directional pseudo-normality of $F_1$ and $F_2$ holds at $(\bar{x},\bar{y}_1,\bar{y}_2)$ in the direction $u$.
\end{enumerate}

Then, $F$ is metrically subregular at $\big(\bar{x}, (\bar{y}_1, \bar{y}_2)\big)$ in the direction $u$. 
\end{theorem}

\begin{proof}
We proceed by contradiction. Suppose the assertion is false.  
By the negation of the directional metric subregularity of $F = (F_1, F_2)$ at $\bar{x}$ in the direction $u$, there exist sequences $t_k \downarrow 0$, $u_k \to u$, and $x_k := \bar{x} + t_k u_k$ such that 
\begin{equation} \label{eq:proof_contradiction_limit}
    \lim_{k \to \infty} \frac{\operatorname{dist}\big(x_k, F^{-1}(\bar{y}_1, \bar{y}_2)\big)}{\operatorname{dist}\big(\bar{y}_1, F_1(x_k)\big) + \operatorname{dist}\big(\bar{y}_2, F_2(x_k)\big)} = \infty.
\end{equation}

\textbf{Step 1: Primal sequences.} 
Noting that $\bar{y}_1 \in F_1(\bar{x})$ and $\bar{y}_2 \in F_2(\bar{x})$, we readily have $\operatorname{dist}\big(x_k, F^{-1}(\bar{y}_1, \bar{y}_2)\big) \le \|x_k - \bar{x}\| \le (\|u\| + 1)t_k$ for all sufficiently large $k \in \mathbb{N}$. Substituting this primal upper bound into \eqref{eq:proof_contradiction_limit} yields: 
\begin{equation} \label{eq:superlinear_rate}
    \lim_{k \to \infty} \frac{\operatorname{dist}\big(\bar{y}_1, F_1(x_k)\big) + \operatorname{dist}\big(\bar{y}_2, F_2(x_k)\big)}{t_k} = 0.
\end{equation}
By condition (i), $F_1$ is metrically subregular at $(\bar{x}, \bar{y}_1)$ in the direction $u$. Hence, there exists $\kappa_1 > 0$ such that, for all sufficiently large $k$, we can select a pull-back point $\check{x}_k \in F_1^{-1}(\bar{y}_1)$ satisfying the approximate distance estimate
\begin{equation} \label{eq:pullback_estimate}
    \|\check{x}_k - x_k\| \le 2\kappa_1 \, \operatorname{dist}\big(\bar{y}_1, F_1(x_k)\big).
\end{equation}
Invoking \eqref{eq:superlinear_rate}, we deduce that $\lim_{k\to\infty} \|\check{x}_k - x_k\|/t_k = 0$. Utilizing the triangle inequality and $(x_k - \bar{x})/t_k \to u$, the directional alignment of this pull-back sequence is preserved: $(\check{x}_k - \bar{x})/t_k \to u$. 

\textbf{Step 2: Variational sequences.}
Set $S_i:=F_i^{-1}(\bar{y}_i)$, $i=1,2$, and $S:=S_1\cap S_2$. Let $L_2$ be a directional Lipschitz modulus of $\varphi(x):=\operatorname{dist}(\bar{y}_2,F_2(x))$, 
and fix a constant $c>2\kappa_1L_2$. Indeed, after shrinking the directional neighborhood if necessary, the directional Lipschitz-like inclusion for $F_2$ implies
\[
    |\varphi(x)-\varphi(z)|\le L_2\|x-z\|
\]
for all $x,z$ in that neighborhood. On the closed set $\operatorname{gph}F_1$, consider
\[
    \Phi(x,y_1):=\varphi(x)+c\|y_1-\bar{y}_1\|.
\]
At the point $(\check x_k,\bar{y}_1)\in\operatorname{gph}F_1$, put $d_k:=\operatorname{dist}(\check x_k,S),\; \varepsilon_k:=\Phi(\check x_k,\bar{y}_1)=\varphi(\check x_k)$. 

We must have $\varepsilon_k > 0$ for all sufficiently large $k$. Otherwise, along a subsequence with $\varepsilon_k=0$, we would have $\check{x}_k \in F^{-1}(\bar{y}_1,\bar{y}_2)$ and hence
\[
    \operatorname{dist}\big(x_k,F^{-1}(\bar{y}_1,\bar{y}_2)\big)
    \le 2\kappa_1\,\operatorname{dist}\big(\bar{y}_1,F_1(x_k)\big).
\]
This would bound the ratio in \eqref{eq:proof_contradiction_limit} by $2\kappa_1$, a contradiction.

Furthermore, \eqref{eq:superlinear_rate} gives $\operatorname{dist}\big(\bar{y}_2, F_2(x_k)\big) = o(t_k)$. Since $F_2$ is directionally Lipschitz-like and $\|\check{x}_k-x_k\|=o(t_k)$, we obtain
\[
    \varepsilon_k
    \le \operatorname{dist}\big(\bar{y}_2,F_2(x_k)\big)
       +L_2\|\check{x}_k-x_k\|
    =o(t_k).
\]
The estimates above and \eqref{eq:proof_contradiction_limit} also give $\frac{d_k}{\varepsilon_k}\to\infty,\; \frac{\varepsilon_k}{t_k}\to0$. 
Indeed, $|d_k-\operatorname{dist}(x_k,S)|\le\|\check x_k-x_k\|$, whereas $\varepsilon_k\le
    \operatorname{dist}(\bar{y}_2,F_2(x_k))
    +L_2\|\check x_k-x_k\|$. Thus, for $D_k:=\operatorname{dist}(x_k,S)$, the failure relation gives $\operatorname{dist}(\bar{y}_1,F_1(x_k))+\operatorname{dist}(\bar{y}_2,F_2(x_k))=o(D_k)$, whence $d_k=D_k+o(D_k)$ and $\varepsilon_k=o(D_k)$.
Moreover, $d_k\le\|\check x_k-\bar{x}\|=O(t_k)$. Since
$\check x_k\in S_1$ and $\varepsilon_k=\varphi(\check x_k)>0$, we have
$\check x_k\notin S$ and hence $d_k>0$ for all sufficiently large $k$.
Define $\lambda_k:=\sqrt{\varepsilon_kd_k},\;
    q_k:=\frac{\varepsilon_k}{\lambda_k}
        =\sqrt{\frac{\varepsilon_k}{d_k}}$.  Then
\begin{equation}\label{eq:ekeland_parameter_limits}
    \frac{\lambda_k}{d_k}
        =\sqrt{\frac{\varepsilon_k}{d_k}}\to0,
    \qquad
    q_k\to0,
    \qquad
    \frac{\lambda_k}{t_k}
        =\sqrt{\frac{\varepsilon_k}{t_k}\frac{d_k}{t_k}}\to0.
\end{equation}

Since $\Phi\ge0$ on $\operatorname{gph}F_1$ and
$\Phi(\check x_k,\bar{y}_1)=\varepsilon_k$, Ekeland's variational
principle, applied with the parameters $\varepsilon_k$ and $\lambda_k$
to $\Phi$ on $\operatorname{gph}F_1$, yields a point
$(\tilde x_k,\tilde y_{1,k})\in\operatorname{gph}F_1$ with the following
three properties:

\begin{equation}\label{eq:ekeland_proximity}
    \|\tilde x_k-\check x_k\|
       +\|\tilde y_{1,k}-\bar{y}_1\|\le\lambda_k;
\end{equation}
\begin{equation}\label{eq:ekeland_value}
    \Phi(\tilde x_k,\tilde y_{1,k})
       \le \Phi(\check x_k,\bar{y}_1)=\varepsilon_k;
\end{equation}
for every
$(x,y_1)\in\operatorname{gph}F_1\setminus
\{(\tilde x_k,\tilde y_{1,k})\}$,
\begin{equation}\label{eq:ekeland_variational_inequality}
\begin{aligned}
    \Phi(\tilde x_k,\tilde y_{1,k})
    <{}\Phi(x,y_1)+q_k\bigl(\|x-\tilde x_k\|+\|y_1-\tilde y_{1,k}\|\bigr).
\end{aligned}
\end{equation}

We record separately the consequences of these three conclusions. First,
\eqref{eq:ekeland_proximity}, \eqref{eq:ekeland_parameter_limits}, and
$(\check x_k-\bar{x})/t_k\to u$ imply $\frac{\tilde x_k-\bar{x}}{t_k}\to u,\; \frac{\tilde y_{1,k}-\bar{y}_1}{t_k}\to0$. 
Thus the Ekeland point remains in the prescribed directional neighborhood
for all sufficiently large $k$. Second, \eqref{eq:ekeland_value} gives
\begin{equation}\label{eq:ekeland_value_consequences}
    0\le\varphi(\tilde x_k)\le\varepsilon_k=o(t_k),
    \qquad
    c\|\tilde y_{1,k}-\bar{y}_1\|
       \le\varepsilon_k=o(t_k).
\end{equation}
Third, \eqref{eq:ekeland_variational_inequality} says precisely that
$(\tilde x_k,\tilde y_{1,k})$ is a strict minimizer on
$\operatorname{gph}F_1$ of
\[
    (x,y_1)\mapsto
    \Phi(x,y_1)
       +q_k\bigl(\|x-\tilde x_k\|
                  +\|y_1-\tilde y_{1,k}\|\bigr),
\]
which will be used both in the following positivity argument and in the
subsequent fuzzy sum rule.

We next show that $\varphi(\tilde x_k)>0$ for all sufficiently large $k$.
Suppose to the contrary that $\varphi(\tilde x_k)=0$. If
$\tilde y_{1,k}=\bar{y}_1$, then $\tilde x_k\in S$, and
\eqref{eq:ekeland_proximity} contradicts
$\lambda_k<d_k$ for all sufficiently large $k$. If
$\tilde y_{1,k}\ne\bar{y}_1$, directional metric subregularity of $F_1$
gives $z_k\in S_1$ such that
\[
    \|z_k-\tilde x_k\|
    \le2\kappa_1\operatorname{dist}(\bar{y}_1,F_1(\tilde x_k))
    \le2\kappa_1\|\tilde y_{1,k}-\bar{y}_1\|.
\]
Using $(z_k,\bar{y}_1)\in\operatorname{gph}F_1$ as a competitor in
\eqref{eq:ekeland_variational_inequality} and using the directional
Lipschitz continuity of $\varphi$, we obtain
\[
    c\|\tilde y_{1,k}-\bar{y}_1\|
    \le
    \bigl(2\kappa_1L_2+(2\kappa_1+1)q_k\bigr)
    \|\tilde y_{1,k}-\bar{y}_1\|,
\]
which is impossible for all large $k$ because $c>2\kappa_1L_2$ and
$q_k\to0$. Hence $\varphi(\tilde x_k)>0$. 
We may now proceed with the normal construction. Choose $\rho_k>0$ such
that $\rho_k=o(t_k),\;\rho_k<\frac{\varphi(\tilde x_k)}{2(1+L_2)}$. 
Applying the fuzzy sum rule for Fr\'echet subdifferentials
\cite[Theorem~2.33]{Mordukhovich2006variational} to the strict minimum
obtained in \eqref{eq:ekeland_variational_inequality}, we obtain
$(x_{1,k},y_{1,k})\in\operatorname{gph}F_1$ and $x_{2,k}'\in X$, all
within $\rho_k$ of the corresponding Ekeland point, together with
\[
    (x_{1,k}^*,-y_{1,k}^*)\in
    \widehat N\bigl((x_{1,k},y_{1,k});\operatorname{gph}F_1\bigr),
    \qquad
    \xi_{2,k}^*\in\widehat\partial\varphi(x_{2,k}'),
\]
such that
\begin{equation}\label{eq:fuzzy_alignment_limit}
    \|x_{1,k}^*+\xi_{2,k}^*\|
    \le q_k+\rho_k\to0,
    \qquad
    \|y_{1,k}^*\|\le c+q_k+\rho_k.
\end{equation}
Moreover, the choice of $\rho_k$ gives $\varphi(x_{2,k}')
    \ge\varphi(\tilde x_k)-L_2\rho_k>0$. 

Choose a projection $\bar{y}_{2,k}'\in F_2(x_{2,k}')$ of $\bar{y}_2$, which exists by assumption. Since $\varphi(x_{2,k}')>0$, the standard fuzzy marginal rule, obtained by applying the same fuzzy sum rule to the lifted function
\[
    (x,y_2)\mapsto\|y_2-\bar{y}_2\|
       +\delta_{\operatorname{gph}F_2}(x,y_2),
\]
can be applied with errors that are $o(t_k)$ and $o(\varphi(x_{2,k}'))$. It yields $(x_{2,k},y_{2,k})\in\operatorname{gph}F_2$ and $(\widehat x_{2,k}^*,-\widehat y_{2,k}^*)\in\widehat N((x_{2,k},y_{2,k});\operatorname{gph}F_2)$ such that
\begin{align}
    &\|x_{2,k}-x_{2,k}'\|+\|y_{2,k}-\bar{y}_{2,k}'\|=o(t_k), \label{eq:second_graph_proximity}\;\\
    &\|\widehat x_{2,k}^*-\xi_{2,k}^*\|\to0,
      \qquad \|\widehat y_{2,k}^*\|\to1, \label{eq:second_graph_dual_approximation}\\
    &\left\langle
       \frac{\widehat y_{2,k}^*}{\|\widehat y_{2,k}^*\|},
       \frac{y_{2,k}-\bar{y}_2}{\|y_{2,k}-\bar{y}_2\|}
      \right\rangle\to1. \label{eq:second_graph_alignment}
\end{align}
The last two relations follow from the fact that every supporting functional $p\in\widehat\partial\|\,\cdot-\bar{y}_2\|(y)$ at $y\ne\bar{y}_2$ satisfies $\|p\|=1$ and $\langle p,y-\bar{y}_2\rangle=\|y-\bar{y}_2\|$, while the fuzzy evaluation points are chosen within $o(\varphi(x_{2,k}'))$ of the projection. In particular, $y_{2,k}\ne\bar{y}_2$ for all sufficiently large $k$. Since $\varphi$ is $L_2$-Lipschitz, $\|\xi_{2,k}^*\|\le L_2$; hence $\{\widehat x_{2,k}^*\}$ is bounded. By the conic property of the Fr\'echet normal cone, we set
\begin{equation}\label{eq:second_graph_normals}
    x_{2,k}^*:=\frac{\widehat x_{2,k}^*}{\|\widehat y_{2,k}^*\|},
    \qquad
    y_{2,k}^*:=\frac{\widehat y_{2,k}^*}{\|\widehat y_{2,k}^*\|},
\end{equation}
so that $(x_{2,k}^*,-y_{2,k}^*)$ is a Fr\'echet normal to $\operatorname{gph}F_2$ and $\|y_{2,k}^*\|=1$. Equations \eqref{eq:fuzzy_alignment_limit} and \eqref{eq:second_graph_dual_approximation} give $\|x_{1,k}^*+x_{2,k}^*\|\to0$, while \eqref{eq:second_graph_alignment} yields the distance alignment and, in particular, $\langle y_{2,k}^*,y_{2,k}-\bar{y}_2\rangle>0$ 
for all sufficiently large $k$. Finally,
\eqref{eq:ekeland_value}, $\rho_k=o(t_k)$, and
\eqref{eq:second_graph_proximity} show that
$(y_{2,k}-\bar{y}_2)/t_k\to0$.

\textbf{Step 3: Dual conclusion.}
The proximity estimates give \eqref{eq:sequential_qualification_primal}. Lemma~\ref{lem:normal_cone_G_representation} then yields the decoupled inclusions \eqref{eq:sequential_decoupled_normals}. Equations \eqref{eq:fuzzy_alignment_limit}--\eqref{eq:second_graph_normals} give the normal balance, the boundedness of $\{y_{1,k}^*\}$, and the normalization of $\{y_{2,k}^*\}$; and \eqref{eq:second_graph_alignment} gives the strict pairing in \eqref{eq:joint_directional_pseudo_normality}. Thus the constructed sequences violate joint directional pseudo-normality, contradicting condition~(ii).  
\end{proof}

The following corollary gives the special case where \(Y_2\) is
finite-dimensional.
\begin{corollary}\label{coro:directional_mscq_intersection_finite}
Let $X$ and $Y_1$ be Asplund spaces, let $Y_2=\mathbb R^d$. 
Let $(\bar{x}, \bar{y}_1, \bar{y}_2)$ be a reference point satisfying $\bar{y}_1 \in F_1(\bar{x})$ and $\bar{y}_2 \in F_2(\bar{x})$. Suppose that $u \in X$ satisfies \(0\in DF_i(\bar{x}\mid \bar{y}_i)(u)\) for $i=1,2$. Assume that the following conditions hold:

\begin{enumerate}
\item[(i)] $F_1$ is metrically subregular at $(\bar{x},\bar{y}_1)$ in the direction $u$, and $F_2$ is directionally Lipschitz-like at $(\bar{x}, \bar{y}_2)$ in the direction $u$.

\item[(ii)] Joint directional quasi-normality of $F_1$ and $F_2$ holds at $(\bar{x},\bar{y}_1,\bar{y}_2)$ in the direction $u$.
\end{enumerate}

Then, $F$ is metrically subregular at $\big(\bar{x}, (\bar{y}_1, \bar{y}_2)\big)$ in the direction $u$.  
\end{corollary}

\begin{proof}
Because the values of $F_2$ are closed subsets of $\mathbb R^d$, the distance to every nonempty value is attained; local nonemptiness follows from the directional Lipschitz-like property. If metric subregularity failed, the proof of Theorem~\ref{thm:directional_mscq_intersection} would therefore produce the sequences in Definition~\ref{defn:joint_directional_pseudo_normality}(i). Passing to a subsequence, $y_{2,k}^*\to y_2^*$ with $\|y_2^*\|=1$. With the Euclidean norm, the additional alignment relation \eqref{eq:second_graph_alignment} obtained in that proof gives
\[
    \left\|y_{2,k}^*-
    \frac{y_{2,k}-\bar{y}_2}{\|y_{2,k}-\bar{y}_2\|}\right\|^2
    =2-2\left\langle y_{2,k}^*,
    \frac{y_{2,k}-\bar{y}_2}{\|y_{2,k}-\bar{y}_2\|}\right\rangle\to0.
\]
Hence, for every $j$ with $[y_2^*]_j\ne0$, the factors $[y_2^*]_j$ and $[y_{2,k}-\bar{y}_2]_j$ have the same sign for all sufficiently large $k$. The sequences therefore violate joint directional quasi-normality, a contradiction.  
\end{proof}

Theorem~\ref{thm:directional_mscq_intersection} and Corollary~\ref{coro:directional_mscq_intersection_finite} extend the test for directional metric subregularity to Asplund spaces. The sequential qualification rules out multiplier sequences whose primal residuals vanish at a first-order rate, \(\frac{P_i(x_{i,k})-s_{i,k}}{t_k}\to 0\), rather than all sequences satisfying the residual condition \(P_i(x_{i,k})-s_{i,k}\to 0\). 
Corollary~\ref{coro:intersection_mscq_finite_differentiable} below is
a directional version of \cite[Theorem~2]{gfrerer2017new} in which
\(X\) is allowed to be an Asplund space.

\begin{corollary}\label{coro:intersection_mscq_finite_differentiable}
Let $X$ be an Asplund space, and let $Y_1,Y_2$ be finite-dimensional. 
Let $P_i: X \to Y_i$ be continuously Fr\'echet differentiable at $\bar{x}$ and let $\Lambda_i \subset Y_i$ be closed for $i = 1, 2$, with $\bar{y}_i:=0\in F_i(\bar{x}):=P_i(\bar{x})-\Lambda_i,\; i=1,2$. 
Let $u \in X$ satisfy $\nabla P_i(\bar{x})(u)\in T(P_i(\bar{x});\Lambda_i)$ for $i=1,2$. Assume that $F_1$ is metrically subregular at $(\bar{x},0)$ in the direction $u$ and that
\begin{equation*}
    \left.
    \begin{aligned}
        &[\nabla P_1(\bar{x})]^*y_1^*
        +[\nabla P_2(\bar{x})]^*y_2^*=0,\\
        &y_i^*\in
        N\bigl(P_i(\bar{x});\Lambda_i;\nabla P_i(\bar{x})u\bigr),
        \qquad i=1,2,
    \end{aligned}
    \right\}
    \quad\Longrightarrow\quad y_2^*=0.
\end{equation*}
Then $F$ is metrically subregular at
\(\bigl(\bar{x},(0,0)\bigr)\) in the direction $u$.
\end{corollary}

\begin{proof}
It suffices to verify joint directional quasi-normality. Suppose, to the contrary, that it fails. Since \(Y_1\) and \(Y_2\) are finite-dimensional, the multiplier sequences in
Definition~\ref{defn:joint_directional_pseudo_normality}(ii) admit a subsequence such that $y_{i,k}^*\to y_i^*,\; i=1,2,\; \|y_2^*\|=1$. 
The sequential normal relations and the continuous differentiability
of \(P_i\) yield
\[
    y_i^*\in
    N\bigl(P_i(\bar{x});\Lambda_i;\nabla P_i(\bar{x})u\bigr),
    \qquad
    x_{i,k}^*=[\nabla P_i(x_{i,k})]^*y_{i,k}^*,
    \qquad i=1,2.
\]
Thus, $0=\lim_{k\to\infty}(x_{1,k}^*+x_{2,k}^*)=[\nabla P_1(\bar{x})]^*y_1^*+[\nabla P_2(\bar{x})]^*y_2^*$. 
The assumed implication gives \(y_2^*=0\), contradicting
\(\|y_2^*\|=1\). Hence joint directional quasi-normality holds, and the
conclusion follows from
Corollary~\ref{coro:directional_mscq_intersection_finite}.
 
\end{proof}

\section{Directional Optimality Conditions via Joint Directional Metric Subregularity}\label{sec:applications}
Based on these foundations, we now derive directional necessary
optimality conditions for \eqref{eq:ocpec}. Since
\(F_i=P_i-\Lambda_i\), feasibility of \(\bar{x}\) means
\(0\in F_i(\bar{x})\); throughout this section, we therefore set
\(\bar{y}_i:=0\) for \(i=1,2\). For \(P:=(P_1,P_2)\) and
\(\Lambda:=\Lambda_1\times\Lambda_2\), we have the following theorem.

\begin{theorem}\label{thm:optimality_condition_ocpec}
Suppose that \(\bar{x}\) is a local minimizer of
\eqref{eq:ocpec}. Let \(u\) be a critical direction of \eqref{eq:ocpec} at \(\bar{x}\).
Suppose that each \(P_i\) is weak\(^*\) strictly Lipschitzian at \(\bar{x}\) and Hadamard directionally differentiable there in the direction \(u\). Write $v_i:=(P_i)'_H(\bar{x};u),\; i=1,2$ (i.e., $(v_1,v_2)=P'_H(\bar{x};u)
    \in T\bigl(P(\bar{x});\Lambda\bigr),\;J'_+(\bar{x};u)\leq0$). 

Assume further that the following conditions hold:

\begin{enumerate}
    \item[(i)]
    \(F_1\) is metrically subregular at
    \((\bar{x},0)\) in the direction \(u\);

    \item[(ii)]
    \(F_2\) is directionally Lipschitz-like at
    \((\bar{x},0)\) in the direction \(u\), and
    \(\operatorname{dist}(0,F_2(x))\) is attained whenever \(x\) belongs to a sufficiently small directional neighborhood
    of \(\bar{x}\) along \(u\);

    \item[(iii)]
    joint directional pseudo-normality or quasi-normality holds at
    \((\bar{x},0,0)\) in the direction \(u\).
\end{enumerate}

Then there exists
$(z_1^*,z_2^*)\in N\bigl(P_1(\bar{x});\Lambda_1;v_1\bigr)
    \times N\bigl(P_2(\bar{x});\Lambda_2;v_2\bigr)$ such that
\begin{equation}\label{eq:joint_optimality_formula}
    0\in
    \partial J(\bar{x};u)
    +\partial\langle z_1^*,P_1\rangle(\bar{x};u)
    +\partial\langle z_2^*,P_2\rangle(\bar{x};u).
\end{equation}
\end{theorem}

\begin{proof}
    By the standing assumptions and Theorem~\ref{thm:directional_mscq_intersection}, the joint constraint mapping $F := (F_1, F_2)$ is metrically subregular at $\bigl(\bar{x},(0,0)\bigr)$ in the direction $u$. This is the setting of Theorem~\ref{thm:further_optimality_mscq} with $G=F$. Set $P := (P_1, P_2)$, $v := (v_1, v_2)$, and $\Lambda := \Lambda_1 \times \Lambda_2$. Theorem~\ref{thm:further_optimality_mscq} then yields a multiplier $z^* := (z_1^*, z_2^*) \in N\big(P(\bar{x}); \Lambda; v\big)$ such that
    \begin{equation*}
        0 \in \partial J(\bar{x}; u) + D_N^* P\big(\bar{x}; (u, v)\big)(z^*) = \partial J(\bar{x}; u) + \partial \langle z^*, P \rangle (\bar{x}; u).
    \end{equation*}
    By virtue of the product rule for directional limiting normal cones established in Proposition~\ref{prop:dir_normal_product}, we have the inclusion $N\big(P(\bar{x}); \Lambda; v\big) \subset N\big(P_1(\bar{x}); \Lambda_1; v_1\big) \times N\big(P_2(\bar{x}); \Lambda_2; v_2\big)$. Furthermore, since $P_1$ and $P_2$ are locally Lipschitz around $\bar{x}$, the component functions $\langle z_i^*, P_i \rangle$ are locally Lipschitz, allowing us to apply the directional subdifferential sum rule (see, e.g., \cite[Theorem~2.13]{mao2025directional}). This inclusion-based calculus guarantees that the joint subdifferential satisfies $\partial \langle z^*, P \rangle (\bar{x}; u) \subset \partial \langle z_1^*, P_1 \rangle (\bar{x}; u) + \partial \langle z_2^*, P_2 \rangle (\bar{x}; u)$. Combining these inclusions yields the desired relation in \eqref{eq:joint_optimality_formula}.  
\end{proof}

\begin{remark}
     In many applications, the mappings $P_1$ and $P_2$ need not simultaneously satisfy strict Lipschitz or Hadamard directional differentiability assumptions. For instance, the value-function reformulation of a bilevel program typically introduces a nonsmooth constraint characterized by the lower-level value function \cite{ye1995necessary,Ye1995,ye1997optimal}. In that specific setting, $Y_2 = \mathbb{R}$, rendering such differentiability conditions unnecessary. 
\end{remark}

The following example illustrates the finite-dimensional differentiable
criterion in Corollary~\ref{coro:intersection_mscq_finite_differentiable}
for a value-function reformulation of a bilevel program. Although
ordinary metric subregularity fails, the corollary yields directional
metric subregularity.

\begin{example}\label{ex:bilevel_finite_directional_mscq}
Let \(H:=\ell^2\), and let \(e_1\) be its first canonical basis
vector. Consider the variable
\(z=(x,y,w)\in Z:=\mathbb R\times\mathbb R\times H\), where \(Z\) is
endowed with its Hilbert product norm, and consider the bilevel problem
\begin{equation}\label{eq:bilevel_joint_example}
    \begin{aligned}
        \min_{\substack{x,y\in\mathbb R\\w\in H}}\quad
        &J(x,y,w):=x^2-y+\langle e_1,w\rangle+\|w\|^2,\\
        \text{s.t.}\quad
        &y\in S(x),\qquad y=0,\qquad \langle e_1,w\rangle=0,
    \end{aligned}
\end{equation}
where \(S(x)\) is the solution mapping of the lower-level problem $\min_{\eta\leq0}(\eta-x)^2$. 

The lower-level solution mapping and value function are, respectively,
\[
    S(x)=
    \begin{cases}
        \{x\}, & x\leq0,\\
        \{0\}, & x>0,
    \end{cases}
    \qquad
    V(x)=(x_+)^2,
\]
where \(x_+:=\max\{x,0\}\). Define $h(x,y):=(y-x)^2-V(x)=(y-x)^2-(x_+)^2$. The value-function reformulation of the bilevel constraints is
\[
    y=0,\qquad y\leq0,\qquad h(x,y)\leq0,\qquad
    \langle e_1,w\rangle=0.
\]
Indeed, for every \(y\leq0\), the definition of \(V\) gives
\(h(x,y)\geq0\); hence \(h(x,y)\leq0\) holds if and only if
\(y\in S(x)\). Set
\[
    \begin{aligned}
        &P_1(z):=(y,y,h(x,y)),
        &&\Lambda_1:=\{0\}\times\mathbb R_-\times\mathbb R_-,\\
        &P_2(z):=\langle e_1,w\rangle,
        &&\Lambda_2:=\{0\}.
    \end{aligned}
\]
Thus, with \(F_i:=P_i-\Lambda_i\), \(i=1,2\), and
\(F:=(F_1,F_2)\), problem \eqref{eq:bilevel_joint_example} is exactly
\[
    \min_{z\in Z}J(z)
    \qquad\text{subject to}\qquad
    P_1(z)\in\Lambda_1,\quad P_2(z)\in\Lambda_2.
\]
Moreover, $F^{-1}(0,0)=\mathcal F=\{(x,0,w)\mid x\geq0,\ \langle e_1,w\rangle=0\}$. Consequently, \(\bar{z}:=(0,0,0)\) is a strict local minimizer because
\[
    J(x,0,w)=x^2+\|w\|^2>0=J(\bar{z})
    \qquad\text{for every }(x,0,w)\in\mathcal F\setminus\{\bar{z}\}.
\]

We first show that ordinary metric subregularity fails at \(\bar{z}\). For
\(t>0\), let \(z_t:=(-t,0,0)\). Since $P_1(z_t)=(0,0,t^2),\; P_2(z_t)=0$, we have
\[
    \operatorname{dist}\bigl(z_t,F^{-1}(0,0)\bigr)=t,
    \qquad
    \operatorname{dist}\bigl(0,F_1(z_t)\bigr)=t^2,
    \qquad
    \operatorname{dist}\bigl(0,F_2(z_t)\bigr)=0.
\]
It follows that
\[
    \frac{\operatorname{dist}\bigl(z_t,F^{-1}(0,0)\bigr)}
    {\operatorname{dist}\bigl(0,F_1(z_t)\bigr)
     +\operatorname{dist}\bigl(0,F_2(z_t)\bigr)}
    =\frac1t\longrightarrow\infty.
\]
Therefore, ordinary metric subregularity fails at \((\bar{z},(0,0))\).

Next, consider the direction \(\bar{u}:=(1,0,0)\). We verify the
assumptions of
Corollary~\ref{coro:intersection_mscq_finite_differentiable}. The image
spaces \(Y_1=\mathbb R^3\) and \(Y_2=\mathbb R\) are
finite-dimensional. Moreover, \(P_1\) is continuously Fr\'echet
differentiable, whereas \(P_2\) is linear. For
\(d=(a,b,c)\in Z\),
\begin{equation}\label{eq:example_derivatives}
    \nabla P_1(\bar{z})d=(b,b,0),
    \qquad
    \nabla P_2(\bar{z})d=\langle e_1,c\rangle.
\end{equation}
In particular,
\(\nabla P_1(\bar{z})\bar{u}=0\) and
\(\nabla P_2(\bar{z})\bar{u}=0\). Moreover, since
\(\nabla J(\bar{z})(\bar{u})=\langle(0,-1,e_1),(1,0,0)\rangle=0\),
\(\bar{u}\) is a critical direction.

We next verify the directional metric subregularity of \(F_1\). Its
solution set is $F_1^{-1}(0)=\{(x,0,w)\mid x\geq0,\ w\in H\}$. 
Every \(z=(x,y,w)\) in a sufficiently small $\mathcal{V}(\bar{z};\bar{u})$ satisfies \(x>0\). Hence $\operatorname{dist}\bigl(z,F_1^{-1}(0)\bigr)=|y|$. 

Since the first component of \(P_1\) represents the equality
constraint \(y=0\),
\[
    \operatorname{dist}\bigl(0,F_1(z)\bigr)
    =\operatorname{dist}\bigl(P_1(z),\Lambda_1\bigr)
    \geq |y|.
\]
Thus, \(F_1\) is metrically subregular at \((\bar{z},0)\) in the
direction \(\bar{u}\). Since
both directional image vectors in \eqref{eq:example_derivatives} are
zero,
\[
    N\bigl(P_1(\bar{z});\Lambda_1;\nabla P_1(\bar{z})\bar{u}\bigr)
       =\mathbb R\times\mathbb R_+\times\mathbb R_+,
    \qquad
    N\bigl(P_2(\bar{z});\Lambda_2;\nabla P_2(\bar{z})\bar{u}\bigr)
       =\mathbb R.
\]
Let \(y_1^*:=(\alpha,\beta,\gamma)\) with
\(\beta,\gamma\geq0\), and let \(y_2^*:=\delta\in\mathbb R\). By
\eqref{eq:example_derivatives},
\[
    [\nabla P_1(\bar{z})]^*y_1^*
       =(0,\alpha+\beta,0),
    \qquad
    [\nabla P_2(\bar{z})]^*y_2^*
       =(0,0,\delta e_1).
\]
Therefore, $[\nabla P_1(\bar{z})]^*y_1^*
    +[\nabla P_2(\bar{z})]^*y_2^*=0
    \; \Longrightarrow\;
    \alpha+\beta=0,\; \delta=0$, 
and hence \(y_2^*=0\), as required by the corollary. We conclude from
Corollary~\ref{coro:intersection_mscq_finite_differentiable} that
\(F=(F_1,F_2)\) is metrically subregular at
\((\bar{z},(0,0))\) in the direction \(\bar{u}\).

Finally, the directional optimality condition can be verified explicitly. Choose $z_1^*:=(1,0,0)
    \in N\bigl(P_1(\bar{z});\Lambda_1;0\bigr),
    \;
    z_2^*:=-1
    \in N\bigl(P_2(\bar{z});\Lambda_2;0\bigr)$. Since
\[
    \begin{aligned}
        \nabla J(\bar{z})=(0,-1,e_1),\;\nabla\langle z_1^*,P_1\rangle(\bar{z})=(0,1,0),\;
        \nabla\langle z_2^*,P_2\rangle(\bar{z})=(0,0,-e_1),
    \end{aligned}
\]
we obtain $0=\nabla J(\bar{z})
     +\nabla\langle z_1^*,P_1\rangle(\bar{z})
     +\nabla\langle z_2^*,P_2\rangle(\bar{z})$. 
\end{example}

\section{Conclusions}\label{sec:conclusions}
We have developed directional necessary optimality conditions for constrained programs in Asplund spaces. Under directional metric subregularity, weak\(^*\) strict Lipschitz continuity, and Hadamard directional differentiability, the analysis yields a scalarized directional optimality condition. Directional metric subregularity is verified through sequential sufficient conditions and directional pseudo-normality/quasi-normality, with directional partial sequential normal compactness used to control vanishing weak\(^*\) limits. For joint constraints, a sequential qualification of directional pseudo-normality type yields directional metric subregularity of the coupled system, from which an optimality condition follows.

\vskip 6mm

\noindent

\bibliographystyle{plain}
\bibliography{sn-bibliography}

@article{gfrerer2013directional,
  author    = {Gfrerer, H.},
  title     = {On directional metric regularity, subregularity and optimality conditions for nonsmooth mathematical programs},
  journal   = {Set-Valued Var. Anal.},
  volume    = {21},
  number    = {2},
  pages     = {151--176},
  year      = {2013}
}

@article{bai2019directional,
  author    = {Bai, K. and Ye, J. J. and Zhang, J.},
  title     = {Directional quasi-/pseudo-normality as sufficient conditions for metric subregularity},
  journal   = {SIAM J. Optim.},
  volume    = {29},
  number    = {4},
  pages     = {2625--2649},
  year      = {2019}
}

@article{benko2019calculus,
  author    = {Benko, M. and Gfrerer, H. and Outrata, J. V.},
  title     = {Calculus for directional limiting normal cones and subdifferentials},
  journal   = {Set-Valued Var. Anal.},
  volume    = {27},
  number    = {3},
  pages     = {713--745},
  year      = {2019}
}

@book{Mordukhovich2006variational,
  author    = {Mordukhovich, B. S.},
  title     = {Variational Analysis and Generalized Differentiation. {I}: Basic Theory},
  series    = {Grundlehren der mathematischen Wissenschaften},
  volume    = {330},
  publisher = {Springer},
  address   = {Berlin},
  year      = {2006}
}

@article{Bai2023Directional,
  author    = {Bai, K. and Ye, J. J.},
  title     = {Directional subdifferential of the value function},
  journal   = {Commun. Optim. Theory},
  pages     = {1--36},
  year      = {2023}
}

@article{long2017calculusb,
  author    = {Long, P. and Wang, B. and Yang, X.},
  title     = {Calculus of directional coderivatives and normal cones in {Asplund} spaces},
  journal   = {Positivity},
  volume    = {21},
  number    = {3},
  pages     = {1115--1142},
  year      = {2017}
}

@article{Ye1995,
  author    = {Ye, J. J. and Zhu, D.},
  title     = {Optimality conditions for bilevel programming problems},
  journal   = {Optimization},
  volume    = {33},
  number    = {1},
  pages     = {9--27},
  year      = {1995}
}

@article{Ye2010,
  author    = {Ye, J. J. and Zhu, D.},
  title     = {New necessary optimality conditions for bilevel programs by combining the {MPEC} and value function approaches},
  journal   = {SIAM J. Optim.},
  volume    = {20},
  number    = {4},
  pages     = {1885--1905},
  year      = {2010}
}

@article{ye1997optimal,
  author    = {Ye, J. J.},
  title     = {Optimal strategies for bilevel dynamic problems},
  journal   = {SIAM J. Control Optim.},
  volume    = {35},
  number    = {2},
  pages     = {512--531},
  year      = {1997}
}

@article{ye1995necessary,
  author    = {Ye, J. J.},
  title     = {Necessary conditions for bilevel dynamic optimization problems},
  journal   = {SIAM J. Control Optim.},
  volume    = {33},
  number    = {4},
  pages     = {1208--1223},
  year      = {1995}
}

@article{gfrerer2014metric,
  author    = {Gfrerer, H.},
  title     = {On metric pseudo-(sub)regularity of multifunctions and optimality conditions for degenerated mathematical programs},
  journal   = {Set-Valued Var. Anal.},
  volume    = {22},
  number    = {1},
  pages     = {79--115},
  year      = {2014}
}

@misc{mao2025directional,
  title         = {Directional Subdifferentials of the Value Function in Asplund Spaces},
  author        = {Mao, W. and Ye, J. J.},
  year          = {2026},
  eprint        = {2608.20241},
  archivePrefix = {arXiv},
  primaryClass  = {math.OC},
  url           = {https://arxiv.org/abs/2608.20241}
}

@article{long2017calculus,
  author    = {Long, P. and Wang, B. and Yang, X.},
  title     = {Calculus of directional subdifferentials and coderivatives in {Banach} spaces},
  journal   = {Positivity},
  volume    = {21},
  number    = {1},
  pages     = {223--254},
  year      = {2017}
}

@article{gfrerer2017new,
  author    = {Gfrerer, H. and Ye, J. J.},
  title     = {New constraint qualifications for mathematical programs with equilibrium constraints via variational analysis},
  journal   = {SIAM J. Optim.},
  volume    = {27},
  number    = {2},
  pages     = {842--865},
  year      = {2017}
}

@book{bonnans2013perturbation,
  author    = {Bonnans, J. F. and Shapiro, A.},
  title     = {Perturbation Analysis of Optimization Problems},
  publisher = {Springer},
  address   = {New York},
  year      = {2000}
}

@article{guo2013mathematical,
  author    = {Guo, L. and Ye, J. J. and Zhang, J.},
  title     = {Mathematical programs with geometric constraints in {Banach} spaces: enhanced optimality, exact penalty, and sensitivity},
  journal   = {SIAM J. Optim.},
  volume    = {23},
  number    = {4},
  pages     = {2295--2319},
  year      = {2013}
}

\end{document}